\documentclass[a4paper]{article}
\usepackage{amsmath,amssymb,amsfonts}
\usepackage{enumerate}
\usepackage{fullpage}
\usepackage{graphicx}
\usepackage{subfigure}
\usepackage{url}
\usepackage{algorithm2e}

\usepackage{boxedminipage}
\usepackage[T1]{fontenc}
\usepackage[english]{babel}
\usepackage{color}
\usepackage{epstopdf}

\newcommand{\N}{\mathbb{N}}
\newcommand{\R}{\mathbb{R}}

\newcommand{\ri}{\mathrm{ri}\,}

\newcommand{\diag}{\mathrm{diag}\,}
\newcommand{\conv}{\mathrm{conv}\,}
\newcommand{\distone}{\mathrm{dist}_1}
\newcommand{\midset}{\;|\;}
\def\Expect{\mathop{\bf E{}}}
\def\Prob{\mathbb{P}}

\newcommand{\qed}{$\Box$}

\newtheorem{Definition}{Definition}[section]

\newtheorem{theorem}[Definition]{Theorem}

 \newtheorem{lemma}[Definition]{Lemma}
 \newtheorem{corollary}[Definition]{Corollary}
 \newtheorem{remark}[Definition]{Remark}

\newenvironment{proof}[1][Proof]{\begin{trivlist}
\item[\hskip \labelsep {\bfseries #1}]}
{\end{trivlist}}

\title{Nonhomogeneous Place-Dependent Markov Chains, Unsynchronised AIMD, and Network Utility Maximization}

\author{Fabian R. Wirth%
\thanks{Faculty of Computer Science, University of Passau, Germany, 
({\tt fabian.lastname@uni-passau.de})} 
\and 
Sonja Stuedli%
\thanks{Newcastle University, Newcastle, Australia 
({\tt sonja.stuedli@uon.edu.au})} 
\and 
Jia Yuan Yu\footnotemark[1]%
\thanks{Concordia University, Montreal, Quebec, Canada, 
({\tt jiayuan.yu@concordia.ca})} 
\and 
Martin Corless%
\thanks{School of Aeronautics and Astronautics, Purdue
    University, West Lafayette, IN, USA ({\tt corless@purdue.edu})} 
\and
Robert Shorten\footnotemark[1]%
\thanks{IBM Research Ireland, Damastown Industrial Estate, Mulhuddart,
  Dublin 15, Ireland, 
({\tt  robshort@ie.ibm.com}).}
}

\date{\today}

\begin{document}

\maketitle

\begin{abstract}      
   \let\thefootnote\relax\footnote{This work is supported in part by the EU FP7 project INSIGHT under grant 318225.}
We present a solution of a class of network utility
    maximization (NUM) problems using minimal communication. The
    constraints of the problem are inspired less by TCP-like congestion
    control but by problems in the area of internet of things and related
    areas in which the need arises to bring the behavior of a large group of
    agents to a social optimum. The approach uses only intermittent
    feedback, no inter-agent communication, and no common clock.

    The proposed algorithm is a combination of the classical AIMD
    algorithm in conjunction with a simple probabilistic rule for the agents to
    respond to a capacity signal. This leads to a nonhomogeneous Markov
    chain and we show almost sure convergence of this chain to the social optimum.
\end{abstract}

%\begin{keywords}

\pagestyle{myheadings}
\thispagestyle{plain}
\markboth{~}{AIMD and Optimization}

\section{Introduction}
Recent developments in the context of Smart Grid, Smart Transportation,
and the internet have given rise to a rich set of optimization problems in
which a number of agents collaborate to achieve a social optimum
\cite{clement2009,Putrus,deilami,low2011,CorlKing2016}. For example, collaborative
cruise control systems are emerging in which a group of vehicles on a
stretch of road share information to determine a speed limit that
minimizes fuel consumption subject to some constraint (traffic flow,
pollution constraints, etc.). Other examples of this problem can be found
in several application domains; in the energy literature
\cite{CSIRO2012solarIntermitency,finn09dsm,biegel2013EccOnOff}; in the
electric vehicles literature \cite{IET,Haesen,Fernandez,Demand_Dispatch},
in distributed load control \cite{low2011}; in the study of control
strategies for thermostatically controlled loads, as refrigerators or air
conditioners \cite{angeli2012,ehsan2013EccOnOff,callaway2013EccOnOff},
and of course, in the optimization literature itself
\cite{StanShor10,tassiulas2001ratecontrol}.\newline

Roughly speaking the optimization problems that emerge in such
applications are simple to solve.  Typically, one wishes to minimize a sum
of strictly convex functions of a single variable subject to a linear, or
perhaps, polynomial constraint. It is well known that such problems can
be readily solved by a multitude of methods in a convex optimization
framework \cite{boyd2004convex}. Notwithstanding this fact, solving these
problems in a smart grid or smart transportation framework is
challenging. The difficulties that arise in such environments are due to
several factors.\newline

First one wishes to find solutions which can
be implemented with minimal communication (or even none) between individual
agents, and between the agents and infrastructure. This need arises due to
the fact that many of these problems are massively large scale in nature,
and continuous inter-agent communication would place an undue burden on
the telecommunications infrastructure \cite{low2011} and due to privacy
considerations.\newline

The second difficulty is that the number of agents participating in the
optimization problem is large and time-varying, and each agent's utility
is private and is not communicated to other agents in order to preserve this privacy. An additional further
difficulty arises because agents in such applications typically have
limited actuation capabilities, i.e. limited capabilities of effecting a change in 
their state. For example, in {\em internet of things
  applications} agents can often only influence their behavior by
switching themselves on or off.  Thus, distributed algorithms for solving
large scale optimization problems in which agents with limited actuation
capabilities collaborate to achieve a common goal, is a highly topical
research problem.\newline

Our objective in this paper is to develop algorithms that can be deployed
in such situations. To this end, consider a network of
$n$ agents, each with a state $x_i \in \R$, $i=1,\ldots,n$ representing an
amount of allocated resource. The allocated resource will be updated at
discrete time instances $k= 0,1,2,\ldots$, where in implementations a
common clock is not required. The agents also keep track of their
individual long term average
\begin{equation}
    \label{eq:1}
    \bar{x}_i(k) = \frac{1}{k+1} \sum_{j=0}^k x_i(j) \,.
\end{equation}
We assume an upper bound $C>0$ of the possible use of resources. To each
agent we associate a cost function $f_i:[0, C] \to\R$. We consider the
problem of network-wide optimal allocation, which can be stated as
\begin{equation}
    \begin{aligned}
\text {minimize}_{w_1,\ldots,w_n} & \quad && \sum_{i=1}^n f_i(w_i) %f_i
                                %(\bar{x}_i) 
\\
\text {subject to} & \quad && \sum_{i=1}^n w_i %\bar{x}_i
= C \,, \quad w_i \geq 0 \,,
i=1,\ldots,n \,.
    \end{aligned}
\end{equation}
%subject to the capacity constraint $\sum_{i=1}^n \bar{x}_i= C$ and for all
%$i$, $\bar{x}_i \geq0$. 

Under suitable convexity assumption this optimization problem has a unique
optimal point $w^*$.  We wish to steer the average values $\bar{x}_i$ to
the optimal point, i.e. we are looking for an algorithm such that
\begin{equation*}
    \lim_{k\to  \infty} \bar{x}_i(k) = w_i^*.
\end{equation*}
Further, we wish to do this with as little inter-agent communication as possible, and
with a minimum amount of centralized actuation.  As we shall see
inter-agent communication is not necessary at all to achieve convergence
to the optimum, nor is it necessary to communicate a feedback signal in
the form of a multiplier.  Rather, we will show that it is sufficient to
provide the centralized, one-bit information, that the constraint has been
reached. Furthermore, the conditions for convergence are independent of
network dimension, depending only on the worst agent
in the system. \\

All this can be achieved
using
only the {\em additive-increase multiplicative decrease} (AIMD)
algorithm. Recall that AIMD 
is an algorithm in which agents continuously claim more and more of the
available resource in a gentle fashion until a notification (of a capacity
event) is sent to them that the aggregate amount of available
resource has been exceeded. This is the additive increase (AI) phase of the algorithm. They then reduce their demand on resource by a factor between zero and one.
This is the multiplicative decrease (MD) of the algorithm.  The AI phase
of the algorithm then restarts immediately. In the algorithm each agent will respond to the capacity signal with a certain
probability $\lambda_i$. The key observation is that by choosing
$\lambda_i$ as a function of the long term average $\bar{x_i}(k)$ we may
achieve convergence to the optimal point $w_i^*$ without any communication
besides the capacity signal.
\newline 

To motivate this result, we make use of the following two known
observations. \newline

{\em Observation 1 (Consensus) :} The optimization problem may be
formulated in a Lagrangian framework as follows. We
introduce the Lagrange parameter $\mu\in\R$ and consider 
\begin{eqnarray}
\label{eq:Lagrange}
    H(w_1,...,w_n, \mu) = \sum_{i=1}^n f_i  (w_i) - \mu \left(\sum_{i=1}^n w_i-C\right).
\end{eqnarray}
From the Karush-Kuhn-Tucker (KKT) conditions \cite[Section 5.5.3]{boyd2004convex}, the
following necessary and sufficient condition for optimality can be
obtained by setting all partial derivatives to zero. 
If we assume that the optimal point $w^*$ has only positive entries, then
the inequality constraints $w_i \geq 0$ are not active. In this case it is
easy to see that the multipliers corresponding to the inequality
constraints vanish in the KKT conditions. So under the assumption of
positivity of the optimal point $w^* \in \R^n$, $\mu^* \in \R$ we have
\begin{eqnarray}
\label{eq:KKT}
\mu^* = \frac{\partial f_i }{\partial x_i}(w_i^*) \;\; \quad \forall \; i= 1,\ldots,n.
\end{eqnarray}
In other words,
the system is at optimality when the derivatives of the utility functions
are in consensus. We will show in Lemma~\ref{lem:posoptimalpoint}  
that the assumptions to be imposed for our algorithm imply that the
optimal point is positive and so \eqref{eq:KKT} does indeed characterize
the optimal point.
\newline

{\em Observation 2 (Ergodic behavior) :} It follows from the results in
\cite{shorten2007modelling}, under the assumptions of ergodicity, that the
ergodic limit of a network of AIMD flows is almost surely of the form
\begin{equation}
\label{eq:steadystate}
\lim_{k \rightarrow \infty}   \frac{1}{k+1}  \sum_{j=0}^{k} x_i(j) 
 =  \frac{\Theta} {\lambda_i}\,,
\end{equation}
where $\Theta$ is a network-specific constant and $\lambda_i$ is the steady-state probability that the $i$th AIMD agent responds to 
a notification of a capacity event.\newline 

With these two observations in mind, we can now aim to choose
place-dependent probability functions $\lambda_i(\cdot)$ so that the
equation for the steady state behavior \eqref{eq:steadystate} is
equivalent to the KKT condition \eqref{eq:KKT}. Suppose that, in the
$k$th iteration, each agent responds to a capacity event with probability
\begin{equation}\label{eq:prob-set}
\lambda_i(\bar{x}_i(k)) \ = \ \Gamma  \ \frac{f_i'(\bar{x}_i(k))}{\bar{x}_i(k)}\,,
\end{equation}
where $\bar{x}_i(k)$ is the average of the last $k$ values of $x_i$. Here
$\Gamma$ is a network wide constant chosen to ensure that
$0<\lambda_i(\bar{x}_i) < 1$.  Suppose now that $\bar{x}_i(k) \approx
x_i^*$. Then we can write $\lambda_i(\bar{x}_i(k)) \approx \lambda_i^*$. Provided that for this choice \eqref{eq:steadystate} holds, we obtain for large $k$
\begin{equation}
\lambda_i(\bar{x}_i(k)) \ \approx \ \Gamma \frac{f_i'(\bar{x}_i(k))}{\Theta} \lambda_i(\bar{x}_i(k))
\end{equation}
and so $f_i'(\bar{x}_i(k)) \approx \Theta/\Gamma \approx  f_j'(\bar{x}_j(k))$ for all $i,j$. These are precisely the KKT conditions which in many cases 
are both necessary and sufficient for optimality. \newline

The purpose of this paper is to show that the above intuition is
true. Specifically, with the place-dependent probabilities $\lambda_i(\cdot)$ chosen as
in \eqref{eq:prob-set}, we do indeed have $\bar{x}_i(k) \approx x_i^*$ for large
$k$. Consequently, the {\em AIMD} algorithm can be modified to solve
distributed optimization problems in asynchronous environments in a manner
that is both effective and efficient in terms of communication
overhead. \newline

Our paper is structured as follows. We begin by reviewing a recently proposed switched systems model of AIMD
dynamics and known results on the stochastic stability of this model for fixed probabilities. The main result for fixed probabilities
is that the long term averages converge almost surely and that this limit can be expressed analytically. We start in Section~\ref{sec:prelims} by introducing notation and recalling some facts about the dynamics of stochastic AIMD algorithms. In  Section 3 we present a discussion of a stochastic AIMD algorithm that solves the NUM problem. In Section~4 we introduce two dynamical systems, representing these algorithms. These differ in the choice of the probability laws. We then state the main convergence results.
Related works are discussed in Section~\ref{sec:relworks}.   
In Section~\ref{sec:example} we apply the results to solve the NUM problem for a network of agents. The main proofs are provided in the Appendix. There, we give intermediate results that link the fixed probability case with the place-dependent case. Specifically, we ask to what degree may the place dependent case 
be approximated with the fixed probability case, and over which time intervals. To this end we study the robustness properties of a deterministic 
system that iterates on the expectation operator. These results are then used to establish the main result of the paper.\newline

\begin{remark} {\bf The main contribution}  of this paper is to propose a new technique for solving a NUM problem that can be used in IoT related situations. 
Various techniques exist for solving such problems and we shall enunciate the difference of our method, with respect to these, in the related work 
section later in the paper. However, since AIMD is so closely related to Transmission Control Protocol (TCP), and since TCP can be studied in an optimization framework,
several brief comments are merited at this point to avoid any confusion. First, we do not account for queues in the network, nor do we assume that these give rise 
to a loss process. Also, we do not use ODE's, or fluid like approximations, to model our network. Rather, we study an exact discrete time system that arises in the study of AIMD dynamics, 
without queues, in which losses are governed agents themselves in the network. In other words, and in the language of TCP, 
in our setting losses are generated at the edge, unlike most TCP systems, in which losses are generated at the center via queuing dynamics.
Our principal concern is to determine rules by which agents generate these losses so that certain optimization problems can be solved. 
All together, this gives rise to a ``Markov-like'' system which requires special machinery for its study, and a large part of this paper is
devoted to developing this machinery. 
\end{remark}

\section{Preliminaries}
\label{sec:prelims}
Our starting point is the suite of algorithms that
underpin the Transmission Control Protocol ({\em TCP}) that is used in internet congestion control. A fundamental building block of TCP is the Additive Increase Multiplicative Decrease ({\em AIMD}) algorithm. 
To discuss AIMD in a formal setting some preliminaries are necessary.\newline

\subsection{Notation:} The vector space of real column vectors with $n$ entries is denoted by
$\R^n$ with elements $x =
\begin{bmatrix}
    x_1 & \ldots & x_n
\end{bmatrix}^\top$, where $x^\top$ denotes the transpose of $x$. The
positive orthant $\R^n_+$ is the set of vectors in $\R^n$ with non-negative coordinates.
For $x,y \in \R^n $, we write $x \gg y$ if $x_i > y_i$ for all
$i=1,\ldots,n$. The space of $n \times n$ matrices is denoted by $\R^{n
  \times n}$ and $\R^{n \times n}_+$ is the set of non-negative matrices,
i.e. the set of matrices in which all entries are non-negative. 
The convex
hull of a set $X$ is denoted by $\conv X$; it may be defined as the
smallest convex set containing $X$.\newline

We denote the canonical basis vectors in $\R^n$ by $e_i, i=1,\ldots,n$ and
let $e:= \sum_{i=1}^n e_i$. The standard $1$-norm is defined by $\|x\|_1 =
\sum_{i=1}^n |x_i|$, $x\in \R^n$. The closed ball of radius $\delta$
around $0$ with respect to this norm is denoted by
$\overline{B}_1(0,\delta)$. The distance of a point $x$ to a nonempty set $Z$ with
respect to the $1$-norm is then
\begin{equation*}
    \distone(x,Z) := \inf \{ \|x - z \|_1 \;;\; z \in Z \}\,.
\end{equation*}

The standard simplex $\Sigma$ in $\R^n$ is defined by
\begin{equation*}
    \Sigma := \left\{ x \in \R^n_+ \;\middle|\; \sum_{i=1}^n x_i =1  \right\}\,.
\end{equation*}
We will write $\Sigma_{n}$ if we want to emphasize that we are working in
$\R^n$. Note that we are only interested in dynamics on $\Sigma$. Thus
when we write $\overline{B}_1(0,\delta)$ we will tacitly assume that we consider the
intersection of this ball with $\Sigma$.  The relative interior of
$\Sigma$ is defined by $\ri \Sigma:= \{ x\in \Sigma \midset x_i > 0
, i=1,\ldots,n
\}$. It will be sometimes useful to use the Hilbert metric
$d_H(\cdot,\cdot)$ on $\ri \Sigma$, \cite{Hartfiel}. Recall that it is given by
\begin{equation*}
    d_H(x,y) := \max_i \log(x_i/y_i) - \min_j\log (x_j/y_j)\,,\quad x,y \in \ri \Sigma\,,
\end{equation*}
and makes $(\ri \Sigma,d_H)$ a complete metric space. A ball of radius
$\delta$ with respect to the Hilbert metric is denoted by $B_H(x,\delta)$;
again without further notice, we will understand that $B_H(x,\delta)$ is
the ball contained in $\Sigma$.  For the sake of analysis, it is sometimes
easier to work with the logarithm removed, in which case we consider
\begin{equation*}
    e^{d_H}(x,y) := \frac{\max_{i} x_i/y_i}{\min_{j} x_j/y_j} \,,\quad x,y \in \ri \Sigma\,.
\end{equation*}
Note that for $x_k, x,y,z \in \ri \Sigma$ we have $\|x_k - y \|_1 \to 0$ if
and only if  $d_H(x_k,y) \to 0$ which is in turn equivalent to
$e^{d_H}(x_k,y)\to 1$.  Furthermore,
$d_H(x,y) < d_H(z,y)$ if and only if $e^{d_H}(x,y) < e^{d_H}(z,y)$.\\

\subsection{AIMD algorithms and stochastic matrices} We have
already mentioned that the AIMD algorithm underpins TCP. We shall not describe the TCP
algorithm here. Rather, we refer the interested reader to
\cite{LowPagDoy02,low2002understanding,Sri04,ShoWirLei06} for details of TCP. 
The dynamics of networks of AIMD flows can be described as
 \begin{equation}\label{synchAIMD}
x(k+1) = A(k)x(k), 
\end{equation}
where $A(k)$ is a non-negative column stochastic matrix and $k$
enumerates the capacity events.
 The matrices $A(k)$ belong to a
finite set of matrices $\mathcal{A}$, which we now describe.  Given two
vectors $\alpha \in \ri \Sigma_n, \beta \in (0,1)^n$ we define a set
${\cal A}$ of
%$m:= 2^n -1$ 
$2^n$
matrices as follows. Let
\begin{equation*}
    B := \left\{ \tilde \beta \in \R^n \;\middle|\; \tilde \beta_i \in \{ \beta_i ,1 \}\,, \quad i=1,\ldots,n \right \} %\setminus \{ e \}
\,,
\end{equation*}
which is clearly a set with $2^n$ elements. The {\em set of AIMD matrices}
is then given by 
\begin{equation}
    \label{eq:AIMDmatrix}
    {\cal A} := \left \{ \diag (\tilde \beta) + \alpha (e-\tilde \beta)^\top
\; \middle| \; \tilde \beta \in B \right \}  \,.
\end{equation}
Note that $\alpha (e-\tilde \beta)^\top \in \R^{n \times n}$ as $\alpha\in \R^{n \times 1}$ and $(e-\tilde \beta)^\top \in \R^{1 \times n}$.
Such matrix sets and the dynamics of Markov chains on $\Sigma$ defined by
${\cal A}$ have been studied in
\cite{wirth2006stochastic,shorten2007modelling,CorlKing2016}. 
%In the following we
%assume that some enumeration of the matrices in ${\cal A}$ has been fixed
%and we will refer to the matrices $A_i\in {\cal A}, i=1,\ldots,%m
%2^n$. 
We single out the matrix for which all diagonal entries are below unity and
%also 
use the convention that the matrix $A_1\in {\cal A}$ is defined using $\beta$, that is,
\begin{equation}
\label{eq:A1}
    A_1 = \diag ( \beta) + \alpha (e- \beta)^\top \,.
\end{equation}
Note that $A_1$ is a column stochastic, positive matrix. 
In particular,
$\lambda=1$ is a simple eigenvalue of $A_1$ and it is larger in magnitude than
all the other eigenvalues of $A_1$ by the Perron-Frobenius theorem.
It is easy to see that a corresponding
positive 
%(Perron) 
eigenvector is  
\begin{equation}
    \label{eq:eigvec1}
    z =
    \left[ 
       \displaystyle \frac{\alpha_1}{1-\beta_1} \quad \ldots \quad \frac{\alpha_n}{1-\beta_n}
    \right] ^\top\,.
\end{equation}
AIMD matrices have the property that they leave the subspace 
\begin{equation*}
    V := \{ x \in \R^n \midset e^\top x = 0 \} 
\end{equation*}
invariant, as they are column stochastic.  Finally, we recall the
following fact about the contractive properties of $A_1$ from
\cite{wirth2006stochastic,CorlKing2016}. In the following statement ${A}_{|V}$
denotes the restriction of $A\in {\cal A}$ to the invariant subspace $V$.

\begin{lemma}
    \label{lem:A1contr}
    Let $\alpha \in \ri \Sigma_n, \beta \in (0,1)^n$ and let ${\cal A}$ be
    the corresponding set of AIMD matrices. Then for all $A\in {\cal A}$
    we have $\|{A}_{|V}\| \leq 1$. Also there exists a constant $c\in
    (0,1)$ such that for the matrix $A_1$ defined by \eqref{eq:A1} we have
\begin{equation*}
    \| {A_1}_{|V} \|_1 = c < 1. 
\end{equation*}
\end{lemma}

\subsection{Elementary results on AIMD} The AIMD algorithm is
often studied under the assumption that probabilities are not
place-dependent; namely, the probability that $A(k) = A \in {\cal A}$ is independent
of $k$ and $x(k)$. It shall be useful to refer to this case in the
remainder of the paper and we briefly recall relevant known results here.
Consider a probability distribution $p:{\cal A}\to [0,1], A \mapsto p_A$ on the set
${\cal A}$ of AIMD matrices. This induces a Markov chain on $\Sigma$ by
setting
  \begin{equation}
      \label{eq:AIMDfixed}
      x(k+1) = A(k) x(k)\,,\quad k\in \N\,, \quad x(0) = x_0 \in \Sigma\,,
  \end{equation}
  where $\Prob(A(k) = A) = p_A$ for $A \in {\cal A}$. In
  particular, the sequence of transition matrices $\{ A(k) \}$ is
  IID. In the sequel,
  we will consider the case in which the probabilities $p_A$ are derived
  from individual drop probabilities $\lambda_i$ of the agents. In this
  case, $\lambda_i$ is the probability that in \eqref{eq:AIMDmatrix} we
  have $\tilde{\beta}_i = \beta_i$. In this case, for every $\tilde{\beta}
  \in B$, or equivalently $A\in {\cal A}$, we have
  \begin{equation}
      \label{eq:fixeddropprob}
    \Prob\Bigl(A(k) = \diag (\tilde \beta) + \alpha (e-\tilde \beta)^\top \Bigr)
    = \prod_{\tilde \beta_i = \beta_i} \lambda _i   \prod_{\tilde \beta_i = 1} (1-\lambda_i)\,.
  \end{equation}

% Given the probability distribution $p$ on ${\cal A}$ we can define
%  ``drop probabilities'' for each of the entries $i$ by defining
%  \begin{equation}
%      \label{eq:fixeddropprob}
%      \lambda_i = \sum_{A_{j,ii}<1} p_j\,,
%  \end{equation}
%  which means, intuitively speaking, that $\lambda_i$ is the sum of the
%  probabilities of those matrices in ${\cal A}$ for which $\tilde \beta_i=
%  \beta_i$, i.e. those matrices for which user $i$ sees a drop. 
For the Markov process defined by \eqref{eq:AIMDfixed} it is known from the results in \cite{wirth2006stochastic,shorten2007modelling,CorlKing2016} that
  if $\lambda =
  \begin{bmatrix}
      \lambda_1 & \ldots & \lambda_n
  \end{bmatrix}
\gg 0$, then there is a unique, invariant, probability measure $\pi$ on
  $\Sigma$ for the Markov chain. We denote by $\hat{\Prob}_\lambda$ the probability measure induced on the
  sample space by the assumption of the IID probabilities $\lambda$. Then,
  for every initial state $x_0\in \Sigma$,
  we have 
  \begin{equation}
     % \lim_{k\to \infty} \frac{1}{k+1} \sum_{t=0}^k x(t;x_0) =  
     \lim_{k\to \infty}\bar{x}(k) = 
  \xi_\lambda 
  \quad, \quad \hat{\Prob}_\lambda - \text{almost surely}  \,,
  \end{equation}
  where 
  \begin{equation}
  \fbox{$\displaystyle
\bar{x}(k) :=  \frac{1}{k+1} \sum_{j=0}^k x(j)
$}
  \end{equation}
  and
  \begin{equation}
  \label{eq:expfixedprob}
  \xi_\lambda := \
\frac{1}{\sum_{\ell=1}^n
  \frac{\alpha_\ell}{\lambda_\ell(1-\beta_\ell)}}  \
\left[ 
          \frac{\alpha_1}{\lambda_1(1-\beta_1)} \quad \ldots \quad
 \frac{\alpha_n}{\lambda_n(1-\beta_n)}
\right]^\top   \,.
  \end{equation}
%where $x(\cdot;x_0)$ denotes samples paths satisfying the initial condition $x(0;x_0)=x_0$.
  As almost sure convergence implies convergence in probability this shows
  that for every $x_0 \in \Sigma$ and $\varepsilon,\delta>0$ there exists a $k_0$ such that for
  all $k\geq k_0$ we have
  \begin{equation}
      \label{eq:convprob}
      %\hat{\Prob}_\lambda\left( \left\| \frac{1}{k+1} \sum_{t=0}^k x(t;x_0) - \xi_\lambda \right\|_1 > \delta \right) < \varepsilon \,.
  \hat{\Prob}_\lambda\left( \left\| \bar{x}(k;x_0) - \xi_\lambda \right\|_1 >
    \delta \right) < \varepsilon \,. 
    \end{equation} 
  It will also be useful to have a uniform version of
  \eqref{eq:convprob}. 
  To this end we define
  % the linear averaging operator
  the random matrix
\begin{equation}
\label{eq:barS}
    \bar{S}(k) := \frac{1}{k+1} \sum_{j=0}^k A(j\!-\!1) \cdots A(0)\,,
\end{equation}
with the interpretation that the summand corresponding to $k=0$ is the
identity $I_n$. 
Our interest in this expression lies in the observation
that for any initial condition $x(0) =x_0 \in \Sigma$ we have
\begin{equation}
 %\frac{1}{k+1} \sum_{t=0}^k x(t;x_0)  = \bar{S}(k)x_0\;
 \bar{x}(k;x_0) = \bar{S}(k)x_0 \,;
\end{equation}
hence
    \begin{equation}
        \label{eq:corstep1}
        %\frac{1}{k+1} \sum_{t=0}^k x(t;x_0) - P_\lambda = (\bar{S}(k) - P_\lambda e^\top ) x_0\,.
        \bar{x}(k) - \xi_\lambda =  (\bar{S}(k) - \xi_\lambda e^\top ) x_0
        \,.
    \end{equation}

  \begin{lemma}
\label{lem:uniformfixed}
%Consider the Markov chain \eqref{eq:AIMDfixed} and let $\lambda\gg 0$ as
%defined by \eqref{eq:fixeddropprob} be fixed. 
Consider the random sequence $\{\bar{S}(k)\}_{k \in \N}$ given by (\ref{eq:barS}) where
$\{A(k)\}$ is IID with probabilities given by \eqref{eq:fixeddropprob}.
Then for every
$\varepsilon,\delta>0$ there exists a $k_0\in\N$ such that
\begin{equation}
    \label{eq:averagematrix}
    \hat{\Prob}_\lambda \left( \left\| \bar{S}(k) - \xi_\lambda e^\top \right\|_1 > \delta \right) < \varepsilon 
\end{equation}
for all $k\geq
k_0$.
Hence, considering the corresponding Markov chain \eqref{eq:AIMDfixed},
 there exists a $k_0\in\N$ such that for 
all $x_0 \in \Sigma$  and all $k\ge k_0$
  \begin{equation}
      \label{eq:convprobuniform}
     % \Prob\left( \left\| \frac{1}{k+1} \sum_{t=0}^k x(t;x_0) - P_\lambda \right\|_1 > \delta \right) < \varepsilon \,.
     \hat{\Prob}_\lambda \left( \left\| \bar{x}(k;x_0) - \xi_\lambda
       \right\|_1 > \delta \right) < \varepsilon \,.
  \end{equation}
  \end{lemma}

  \begin{proof}
    As $\Sigma$ contains the canonical basis vectors $e_i$, and the norm
    on $\R^{n \times n}$ induced by $\|\cdot\|_1$ is the  max column sum norm,
    it follows that for all $M\in \R^{n \times n}_+$
    \begin{equation}
        \label{eq:lemnormdef}
 \|M\|_1 = \max_{i=1,\ldots,n} \|Me_i \|_1 = \max \{ \|Mx \|_1 \midset  x \in \Sigma \} \,.
    \end{equation}
    Fix $\varepsilon, \delta>0$. 
    By \eqref{eq:convprob} and in view of
    \eqref{eq:corstep1}, we may choose for each $i$ an integer $k_i$ such that for all $k\geq
    k_i$ we have
     \[
     \hat{\Prob}_\lambda ( \| (\bar{S}(k) - \xi_\lambda e^\top )
    e_i \| > \delta ) < \frac{\varepsilon}{n}\,.
    %\,,\quad i=1,\ldots,n\,.
    \] 
    By choosing $k_0 :=
    \max\{k_1, \ldots, k_n\}$ we thus obtain for all $k\geq k_0$ that
    \begin{equation}
        \label{eq:corstep2}
        \hat{\Prob}_\lambda \left( \ \left\| \bar{S}(k) - \xi_\lambda
        e^\top  \right\|_1 > \delta \right) = 
        \hat{\Prob}_\lambda \left(  \max_{i=1,\ldots,n} \ \left\| \left( \bar{S}(k) - \xi_\lambda
        e^\top \right) e_i \right\|_1 > \delta \right) < \varepsilon\,,
    \end{equation}
    where we have used the standard estimate 
    $\hat{\Prob}_\lambda \left( \cup_{i=1}^n W_i \right) \leq \sum_{i=1}^n \hat{\Prob}_\lambda\left( W_i \right)$
  for events  $W_i$.
The claim \eqref{eq:convprobuniform} is now an immediate consequence of \eqref{eq:lemnormdef}.
  %     By Lemma~\ref{lem:A1contr} all matrices $A_j\in {\cal A}$ satisfy
  %     $\|{A_j}_{|V}\| \leq 1$. Thus the same is true for all products of
  %     matrices with factors in ${\cal A}$. For a particular sequence of
  %     matrices $\omega = \{ A_{j_0}, A_{j_1}, \ldots \}$, denote by
  %     $x(t;x_0,\omega)$ the realization of the Markov process, starting in
  %     $x_0$ and choosing that sequence of matrices.  We then have for all
  %     $x_0,y_0 \in \Sigma$ that $x_0-y_0 \in V$ and so by linearity we
  %     have for all $k\geq 1$ that
  %     \begin{equation}
  %     \label{eq:approxprelim}
  %         \left\| \frac{1}{k+1} \sum_{t=0}^k x(t;x_0,\omega) - \frac{1}{k+1} \sum_{t=0}^k x(t;y_0,\omega) \right\|_1 \leq \| x_0 - z_0 \|_1  \,.
  %     \end{equation}
  %     To set up a compactness argument, fix $\hat{x}_0\in \Sigma$ and
  %     $\varepsilon,\delta>0$.  Appeal to \eqref{eq:convprob} to choose
  %     a constant $k_0=k_0(\hat{x}_0)$ such that for all $k\geq k_0$ we have 
  %       \begin{equation*}
  %     \Prob\left( \left\| \frac{1}{k+1} \sum_{t=0}^k x(t;\hat{x}_0) - P_\lambda \right\|_1 > \frac{\delta}{2} \right) < \varepsilon \,.
  % \end{equation*} 
  % Note that the continuity result in \eqref{eq:approxprelim} is uniform in
  % $k$, and so we obtain that \eqref{eq:convprob} holds with the constants
  % $\varepsilon,\delta$ for all $x_0\in \Sigma$ in a suitably small
  % neighborhood of $\hat{x}_0$ and for all $k\geq k_0$. A standard
  % compactness argument choosing a finite open cover then completes the
  % proof. 
~\hfill\qed
  \end{proof}

\subsection{Some comments on stochastic convergence} The main result of this paper yields conditions for almost sure convergence
of the sample paths of a Markov chain. For the benefit of the reader we
briefly point to relevant parts of the literature, where this notion is
discussed. Readers familiar with notions of stochastic 
convergence may skip this section. \\

Given a Markov chain and an initial condition, we can consider the set of
all possible sample paths $ \{ x(k;x_0,\omega) \midset \omega \in \Omega
\}$, where $\Omega$ is an index set for the set of different sample paths
or trajectories of the Markov chain. In our case, the set $\Omega$ can be
identified with the set of all sequences with values in $\{ 1,\ldots,2^n
\}$, which can be interpreted as the set of sequences $\{ A(0), A(1)
,\ldots \}$ that lead to a particular sample path. Kolmogorov's existence
theorem now states that the marginal probabilities that are induced by the
Markov chain on finite-time intervals define a probability measure
$\Prob_\Omega$ on $\Omega$, the set of sample paths, see \cite[Sections 2,
24, 36]{billingsley2009convergence}.\\

The statement that convergence to a limit happens almost surely, thus
means that the measure of the set of sample paths which are converging is
$1$; with respect to the probability measure $\Prob_\Omega$ on the sample
space. Thus almost sure convergence means that the convergence happens
with probability one, with the right interpretation of the probability
measure.\\

It is furthermore known, that almost sure convergence implies convergence
in probability, \cite{billingsley2009convergence}. The latter concept is
implicitly defined in \eqref{eq:convprobuniform}: for every
$\varepsilon>0$ and $\delta>0$ there exists a $k_0$ such that for all
$k\geq k_0$ the probability of being further away from the limit than
$\delta$ is smaller than $\varepsilon$.

\section{AIMD Based Optimization Algorithm}
\label{sec:optim}

In this section, we formally define the class of distributed optimization
problems that can be addressed using the algorithm presented in this paper.
Recall: let $n\in \N$, $C>0$ and $f_i:[0,C]\to \R$ be strictly convex and
continuously differentiable, $i=1,\ldots, n$ and consider the optimization
problem

\begin{equation}
    \label{eq:optimprob}
    \begin{aligned}
\text{minimize} \quad &        \sum_{i=1}^n f_i (\bar{x}_i)\\
\text{subject to} \quad & \sum_{i=1}^n \bar{x}_i=C \,, \quad \bar{x}_i \geq 0\,.
    \end{aligned}
\end{equation}

We are interested in finding the optimal point $x^\ast$ in which the
minimum is achieved. It is well-known that by compactness of the feasible
space an optimal solution exists; it is unique by the assumption of strict
convexity. We now formulate conditions guaranteeing that the unique,
optimal point $x^\ast$ is characterized by 
the existence of a constant $\mu^*\in \R$ such that 
\begin{equation}
\label{eq:KKTsimple}
    \sum_{i=1}^n x^\ast_i=C, \quad x^\ast_i \geq 0, \quad
    f_i' (x^\ast_i) = \mu^*, \quad i=1,\ldots, n\,.
\end{equation}
These conditions are such that the algorithm which we have briefly
motivated can be implemented.

\begin{lemma}
    \label{lem:posoptimalpoint}
    Let $n\in \N$, $C>0$ and $f_i:[0,C]\to \R$ be strictly convex and
    continuously differentiable, $i=1,\ldots,n$. Assume that there exists a
    constant $\Gamma >0$ such that for all $x
    \in [0,C]$ and all $i=1,\ldots,n$ we have
    \begin{equation}
        \label{eq:lemposoptimalpointcond}
        0 \leq \Gamma \ \frac{f'_i(x)}{x} \leq 1 \,,
    \end{equation}
    where for $x=0$ the condition \eqref{eq:lemposoptimalpointcond} is
    supposed to hold for the continuous extension of the middle term in $x=0$.
    Then
    \begin{enumerate}[(i)]
      \item There exists a unique optimal point $x^*$ for the minimization
        problem \eqref{eq:optimprob}.
      \item The optimal point $x^*$ is positive, i.e. all its entries are
        strictly positive.
      \item The optimal point is characterized by the simplified KKT
        conditions \eqref{eq:KKTsimple}.
    \end{enumerate}
\end{lemma}

\begin{proof}
    (i) This is an immediate consequence of compactness of the feasible
    set and strict convexity of the cost function.

    (ii) By strict convexity of the $f_i$, the derivatives $f_i'$ are
    strictly increasing. Also \eqref{eq:lemposoptimalpointcond} implies
    that $f_i'(0) = 0$ for all $i=1,\ldots,n$ as otherwise $f_i'(x)/x$
    could not be continuously extended to $x=0$. Let  $x_0 =
    \begin{bmatrix}
        x_{0,1} & \ldots & x_{0,n}
    \end{bmatrix}$ be a point on the relative boundary of the
    simplex $\Sigma_n$, i.e., such that one of its entries equals $0$. Choose indices
    $j,\ell$ such that 
    $x_{0,j}=0, x_{0,\ell} >0$. Denote 
    \begin{equation*}
x_\varepsilon :=  \begin{bmatrix}
        x_{\varepsilon,1} & \ldots & x_{\varepsilon,n}
    \end{bmatrix} := x_0 + \varepsilon e_j -
    \varepsilon e_\ell .\end{equation*}
    Clearly, $x_\varepsilon$ satisfies all constraints of
    \eqref{eq:optimprob}
    provided
    $\varepsilon>0$ is small enough. 
    We claim that for $\varepsilon >0$        
    sufficiently small we have $\sum_{i=1}^n f_i(x_{\varepsilon,i}) <
    \sum_{i=1}^n f_i(x_{0,i})$, which shows  that $x_0$ is not an optimal
    point. To prove the claim consider 
    the derivative of the total cost with respect to $\varepsilon$ at $\varepsilon=0$. We have
    \begin{equation}
        \frac{d}{d \varepsilon}_{\varepsilon=0} \sum_{i=1}^n
        f_i(x_{\varepsilon,i}) 
       = f_j'(0)
        - f_\ell'(x_\ell) = - f_\ell'(x_\ell) < - f_\ell'(0) = 0 \,,
    \end{equation}
    where we have used that $f_\ell'$ is strictly increasing. It follows that
    $x_0$ is not an optimal point for the optimization problem
    \eqref{eq:optimprob}. As $x_0$ was arbitrary on the relative boundary
    of the simplex this proves the assertion.

(iii) It follows from (ii) that the optimal point $x^*$ for problem
\eqref{eq:optimprob} is also an optimal point for the optimization problem
\begin{equation}
    \label{eq:optimprob2}
    \begin{aligned}
\text{minimize} \quad &        \sum_{i=1}^n f_i (\bar{x}_i)\\
\text{subject to} \quad & \sum_{i=1}^n \bar{x}_i=C \,, \quad 0< \bar{x}_i
< C\,.
    \end{aligned}
\end{equation}
As the latter optimization problem is defined on an open subset of the
affine space given by the constraint $\sum_{i=1}^n \bar{x}_i=C$, it
follows that an optimal point, if it exists, satisfies the standard
Lagrange optimality conditions. As we have seen in \eqref{eq:Lagrange}, \eqref{eq:KKT} these simplify to the
conditions stated in \eqref{eq:KKTsimple}. 
\end{proof}

The implementation of the algorithm uses the current state of the
users at a given time instance $t$, which we denote by $x_i(t)$ and the
long-term average of the states of the users denoted by
$\overline{x}_i(t)$. It is the aim of the algorithm to obtain convergence
of the long-term averages to the KKT point $x^\ast$.\newline

Furthermore, the actual algorithm implemented on each agent is based on the following
assumptions. We assume that agents can infer when $\sum_{i=1}^n x_i\geq
C$. If this is not the case we assume that each agent is informed of a
constraint violation at the instant of its occurrence using binary
feedback.\footnote{This is the intermittent feedback and the limited
  communication referred to throughout the paper.} This is a capacity event
notification. Upon receipt of such a notification agent $i$ updates
the state $x_i(t)$:
\begin{eqnarray}
x_i(t^+) = \beta x_i(t^-) \nonumber
\end{eqnarray}
with probability 
\begin{equation}
    \label{eq:5}
    \lambda_i(\bar{x}_i(t^-)) = \Gamma
\frac{f_i'(\bar{x}_i(t^-))}{\bar{x}_i(t^-)}.
\end{equation}
  Note that in this
definition we implicitly assume that the assumptions of Lemma
\ref{lem:posoptimalpoint} are met. In particular,  there exists a constant $\Gamma>0$
such that $\lambda_i(\bar{x}_i) \in [0,1]$ for all values $\bar{x}_i \in
[0,C]$. This restricts the admissible choices for the cost functions
$f_i$. We will discuss different ways of treating more general functions
in Remark~\ref{rem:lambdaprops}.

At all other time instants the
rate of change of $\dot{x}_i(t)$ is chosen to be a positive quantity.
Note that the superscripts $+$ and $-$ denote the instants immediately
prior and after a capacity event notification, respectively. This leads to the
following discrete time algorithm that is implemented on each of the
agents. We assume a common time step $h$ is fixed and each agent $i$ has
an internal offset $T_i$. For the sake of abbreviation, we denote $ k_i :=
T_i + kh, {k\in\N}$ and $x_i(k_i):= x_i(T_i+kh)$.

\vspace{5mm}
\begin{algorithm}[H]
  \underline{\bf Initialization:} Each agent sets its state $x_i(0_i)$  to an arbitrary value\;
\hspace*{2.5cm}The parameter $\Gamma$ is broadcast\;
 \While{agent $i$ is active}{
  \eIf{$\sum_{\ell=1}^n x_\ell(k_i)<C$}{
   $x_i((k+1)_i) = x_i(k_i)+\alpha$ \;
   }{
   $x_i((k+1)_i) = \beta x_i(k_i)$ with probability $\lambda_i(\bar{x}_i(k_i)) = \Gamma \frac{f'_i(\bar{x}_i(k_i))}{\bar{x}_i(k_i)}$ and\\
$x_i((k+1)_i) = x_i(k_i)+\alpha$ otherwise\;
  }
 }
 \caption{AIMD algorithm run by each agent \newline}
\end{algorithm}
\vspace{5mm}

It is clear that the performance of the algorithm depends crucially on a
number of assumptions. For example, we have assumed that the time between
sample points is the same for all agents (note that a common clock is not
necessary). Also the algorithm is implemented in discrete time, while the
AIMD model we analyze has an implicit continuity assumption. This
discrepancy requires that the sample times $h$ are sufficiently small,
when compared to $C$ and $n$. We will tacitly assume that the modeling
error due to discretization effect is sufficiently small.
A few further comments are required.  \newline

\begin{remark}
The constant $\Gamma$ is chosen to ensure that each $\lambda_i(\bar{x}_i) \in (0,1)$.  Thus $\Gamma$ depends on the worst utility function and must be communicated 
to all agents prior to the algorithms use. It is a network dependent quantity that is {\bf independent} of network dimension. 
\end{remark}

% \begin{remark}
% In discrete time the algorithm could be implemented as follows:
% At each instant $t_k$ all agents $i$ perform the update
% \begin{equation*}
%     x_i(t_{k+1}) = x_i(t_k) + \alpha h
% \end{equation*}
% except upon notification of capacity; at which point each agent $i$ with probability $\lambda_i(\bar{x_i}(k))$ scales
% \begin{equation*}
%     x_i(t_{k+1}) = \beta x_i(t_k) \,.
% \end{equation*}
% We note that there is a small approximation error when formulating the
% above discrete-time model as an AIMD system.
% \end{remark}

We now briefly discuss how to reformulate NUM problems so that they satisfy the assumptions of our set-up. Note that the following list is not exhaustive.

\begin{remark}
\label{rem:lambdaprops}
    While the assumption that $\lambda_i(r)$ is well defined and in
    $[0,1]$ for all $r\in [0,C]$ might sound restrictive for the problem
    at hand, we note that the following modifications of the problem yield
    a feasible solution.
    \begin{enumerate}[(i)]
      \item In case the objective functions are not increasing, we can
        define the constant
        \begin{equation*}
            q := \max_{i= 1,\ldots,n} \left| f'_i(0) \right| 
        \end{equation*}
        and consider the objective functions $\tilde f_i$, given by
        $\tilde f_i(r) = f_i(r) + q r$, $r\in[0,C]$, which are now strictly
        increasing. Note that this does not change the KKT point, as
        $\tilde f_i' \equiv f_i' + q$, so that the condition that all
        derivatives are equal is met at the same point $x$.
      \item A second concern is that even if $f_i' \geq 0$ on $[0,C]$, the
        expression $f_i(r)/r$ might tend to $\infty$ as $r\to 0$,
        depending on the nature of the derivative of $f_i$ at $0$. In this
        case we may replace \eqref{eq:5} with
        \begin{equation}
            \label{relax}
             \tilde \lambda_i(r) := \min \left\{ 1 , \Gamma \frac{f'(r)}{r} \right \}\,,\quad r \in [0,C]\,. 
        \end{equation} 
        This amounts to a regularization of the optimization problem which we
        briefly outline in a simple situation. Assume that there is a
        unique point $r_\Gamma\in [0,C]$ such that $r_\Gamma = \Gamma
        f'(r_\Gamma)$. If $\lambda_i$ in \eqref{relax} were the result of
        the definition in \eqref{eq:5} the corresponding objective function would be
        \begin{equation}
            \tilde f_i(r) := \left \{
              \begin{matrix}
                  \frac{1}{2 \Gamma}r^2 &\quad & r \in [0,r_\Gamma]\\
                  f_i(r) - f_i(r_\Gamma) + \frac{1}{2}r^2_\Gamma & \quad & r \in [r_\Gamma,1] 
              \end{matrix} \right.,\quad r \in [0,C].  
        \end{equation}
        More generally, there could be several interlacing intervals, in
        which the condition $r_\Gamma < \Gamma f'(r_\Gamma)$ is satisfied
        or not. The important point here is that a decrease of $\Gamma$
        leads to a decrease of $r_\Gamma$, so that by choosing $\Gamma$
        small enough, the KKT point of the original problem will be found
        by the algorithm.
    \end{enumerate}
\end{remark}

We note that  although our analysis will be performed for a fixed
number of agents, this is not necessary in the implementation of the
algorithm. Indeed, as no information is required on the number of
agents, these may join or drop out of the network at any time and
the network will automatically readjust the KKT point given the new set of
agents.

\section{Convergence Analysis}
\label{sec:finiteav}

In this section we discuss two versions of the stochastic algorithm for
the approximation of the KKT point $x^*$. The common feature of these
algorithms is that the probabilities for backing off depend on an average
of past states. In the first version, we assume that there is a fixed
window over which the average is taken, while in the second case the
average is taken over the complete history starting at time $t_0=0$.\\

The two approaches are amenable to different methods of analysis. In the
first case the problem may be recast in terms of a homogeneous Markov
chain with state-dependent probabilities, sometimes also called an
iterated function system (IFS). In this setting classical results ensure
the existence of an attractive invariant measure and ergodicity results
follow,
\cite{barnsley1988invariant,elton1987ergodic,shorten2007modelling}. However,
the real convergence result of interest can be proved for the second
algorithm which only gives rise to a nonhomogeneous Markov chain and for
which the powerful methods that exist for the first case are not
available. The method of proof relies here on a detailed analysis of the
system dynamics using appropriate Lyapunov functions.\\

For convenience, we will assume $C=1$ in the remainder of the paper.
We consider a set of AIMD matrices for fixed additive increase parameter
$\alpha>0$ and multiplicative decrease parameter $\beta\in (0,1)$. We will
assume there are probability functions $\lambda_i :[0,1] \to [0,1]$,
$i=1,\ldots,n$ that are used by each agent to determine the probability of
responding to the intermittent feedback signal, based on an average of past
values of $x$. We assume that these functions
\begin{equation}
    \label{eq:lambdadef}
    \lambda_i : [0,1] \to [0,1], \quad i=1,\ldots,n,
\end{equation}
satisfy the following assumptions
\begin{enumerate}
  \item[(A1)] $\lambda_i$ is continuous, $i=1,\ldots,n$;
  \item[(A2)] $r \mapsto r \lambda_i(r)$ is strictly increasing on $[0,1]$,
    $i=1,\ldots,n$;
  \item[(A3)] There exists a constant $\lambda_{\min}$ such that
    $\lambda_i(r) \geq \lambda_{\min} >0 $ for all $r\in [0,1],
    i=1,\ldots,n$;
\end{enumerate}
Note that these assumptions are satisfied for the choice of probability
functions described in Section~\ref{sec:optim}. In particular, (A2) is a
consequence of convexity.  \footnote{The formulation can be extended using the same arguments to the 
  assumption that the AIMD parameters of the agents are chosen such that
  $\alpha \in \ri \Sigma, \beta \in (0,1)^n$ satisfy the added assumption
  that the quotient $\alpha_i/(1-\beta_i)$ is a constant independent of
  $i$.} We will show in Lemma~\ref{lem:Pfix} that
under the above conditions there is a unique KKT point $x^\ast \in \Sigma$.\\

It is discussed in \cite{ShoWirLei06,wirth2006stochastic} that the
dynamics of an algorithm of the type of Algorithm~1 can be well
approximated by a Markov chain of AIMD matrices. In fact, if we let
$k=0,1,2,\ldots$ be the consecutive labels of the time instances at which
the constraint is met, then the evolution from one constraint event to the
next is given by \eqref{eq:AIMDfw1} below, where $A(k)$ is one of the AIMD
matrices describing the problem.  Note that the probabilities $p_A(\cdot)$ for
the matrices $A \in {\cal A}$ are now determined by the assumption that
the agents act in a stochastically independent manner, so that the
probability of a particular drop pattern encoded in $A\in {\cal A}$ is given by the
product of the probabilities of the individual agents responding or not.\\

The system of interest is given by the iteration of AIMD matrices in the form
\begin{equation}
    \label{eq:AIMDfw1}
    x(k+1) = A(k) x(k)\,,
\end{equation}
where the matrices $A(k)$ are chosen from the set of AIMD matrices ${\cal
  A}$ using a probability distribution that depends on the history of the
sample path.  Specifically, we consider the following two cases:\\

(i) {\bf finite averaging:} We consider a fixed time window of length $T$.
For $k\geq T-1$ consider the average
\begin{equation}
    \label{eq:finitewindow}
    \fbox{$\displaystyle
    \bar{x}_{T}(k) := \frac{1}{T} \sum_{j=0}^{T-1} x(k-j)
    $}
    %\,,\quad k=T,T+1,\ldots\,,
\end{equation}
and suppose  that there are probability functions $p_A :\Sigma \to [0,1]$,
$A\in {\cal A}$ such that
\begin{equation}
    \label{eq:probdef-fw}
    \Prob(A(k) = A) = p_A(\bar{x}_{T}(k))\,.
\end{equation}
(i) {\bf long-term averaging:} 
In this situation, we consider the average
\begin{equation}
    \label{eq:average}
    \fbox{$\displaystyle \bar{x}(k) := \frac{1}{k+1} \sum_{j=0}^k x(j)$}
\end{equation}
for $k=0,1,\ldots$
and suppose  that there are probability functions $p_A :\Sigma \to [0,1]$,
%$j=1,\ldots,2^n$ 
$A\in {\cal A}$ such that
\begin{equation}
    \label{eq:probdef}
    \Prob(A(k) = A) = p_A(\bar{x}(k))\,.
\end{equation}

\subsection{Finite Averaging}

The condition~\eqref{eq:probdef-fw} needs to be interpreted along sample
paths: for each specific realization of the Markov chain, the
probabilities at time $k$ are a function of the average over the
time interval $[k-T+1,\ldots,k]$ for the given realization.  We will model
this Markov chain as a Markov chain with state-dependent probabilities
on the space $\Sigma^T$.  In view of the evolution \eqref{eq:AIMDfw1} with
\eqref{eq:finitewindow}, \eqref{eq:probdef-fw}, define the new variable
\begin{equation}
    \label{eq:finitewinstate}
    z(k) :=
    \begin{bmatrix}
        x(k) & \frac{1}{2} \bigl(x(k) + x(k-1)\bigr) & \ldots &
        \frac{1}{T} \bigl( x(k) + \ldots + x(k-T+1)\bigr)
    \end{bmatrix}^\top \,.
\end{equation}
It is then easy to see that the evolution of $z(k)$ is described by the
Markov chain
\begin{equation}
    \label{eq:MPfinitetime}
    z(k+1) =
    \begin{bmatrix}
        A(k) & 0 & & \ldots & &0 \\
        \frac{1}{2} \left( A(k) + I \right) & 0 & & \\
        \frac{1}{3} A(k) & \frac{2}{3} I & 0 & &&\vdots\\
            \vdots & 0& \ddots & \ddots && \\
            \vdots & & & \ddots & 0  & 0\\
        \frac{1}{T} A(k) & 0 & \ldots & 0 & \frac{T-1}{T} I & 0
    \end{bmatrix}
z(k) =: A_T(k) z(k)\,.
\end{equation}

Given $A\in {\cal A}$, we denote by $A_T \in \R^{Tn \times Tn}$ the matrix
obtained from $A$ through the construction in \eqref{eq:MPfinitetime}.
Note that each matrix $A\in {\cal A}$ uniquely defines a matrix
$A_{T}$ and the set of possible matrices ${\cal A}_{T}$ occurring in the
Markov chain \eqref{eq:MPfinitetime} is defined in this way. The Markov
chain is thus defined with the place-dependent probabilities
\begin{equation}
    \label{eq:placedeplift}
    \Prob \left( A_T(k) = A_{T} \right) = p_A\left(z_T(k)\right)\,,
\end{equation}
where $A \in {\cal A}$ and
where $z_T(k) \in \Sigma$ denotes the $T$th component vector of $z(k)$.
The following norm on $\R^{Tn}$ simplifies the analysis of the Markov
chain considerably, as it reveals its contractive properties. We define
\begin{equation*}
    \|z\|_T := \max_{i=1,\ldots,T} \| z_i\|_1 \,, \quad \text{where} \quad  z =
      \begin{bmatrix}
          z_1^\top & \ldots & z_T^\top
      \end{bmatrix}^\top, z_i \in \R^n , i=1,\ldots,T\,.
\end{equation*}

\begin{lemma}
\label{lem:ATcont}
    \begin{enumerate}[(i)]
      \item For all ${A}_T \in {\cal A}_T$ the matrix norm induced by
        $\|\cdot\|_T$ satisfies
        \begin{equation}
            \label{eq:normT1}
            \| A_T \|_T \leq 1\,.
        \end{equation}
      \item The subspace
        \begin{equation}
            \label{eq:WTdef}
            W := \{ z \in \R^{Tn} \midset e^\top z_i = 0\,, \forall\ i=1,\ldots,T \} 
        \end{equation}
        is invariant under all $A_T \in {\cal A}_T$.
\item %
For all $z\in W$, $A_{T} \in {\cal A}_T$ it holds that
\begin{equation*}
    \|A_{T} z \|_T = \|z\|_T \quad \Rightarrow \quad A z_1=z_1\,.
\end{equation*}
In particular, we have 
        \begin{equation}
            \label{eq:ATestimate}
            \left\| A_{T,1}{}_{|W} \right\|_{T} \leq \frac{c + T-1}{T} <1\,,
        \end{equation}
where $c<1$ is the constant given by Lemma~\ref{lem:A1contr}.
    \end{enumerate}
\end{lemma}

\begin{proof}
    (i) This is a straightforward calculation.

    (ii) This is an easy consequence of $e^\top A = e^\top, A\in {\cal
      A}$.

    (iii) Consider $A\in {\cal A}$ and the corresponding matrix $A_T\in
    {\cal A}_T$. Assume that $\|A_{T} z\|_T= \|z\|_T$ and consider an index $i$
    such that $\|z\|_T > \|z_i\|_1$. As $\|A\|_1 \leq 1$, it follows that 
    \begin{equation*}
        \|(A_{T} z)_{i+1}\|_1 =
        \left\|\frac{1}{i+1} A z_1 + \frac{i}{i+1} z_i
        \right\|_1 < \|z\|_T \,.
    \end{equation*}
    Thus if $\|A_{T} z\|_T= \|z\|_T$ then necessarily $\| (1/(i+1)) A z_1 +
    (i/(i+1)) z_{i} \|_1 = \|z_i\|_1$ for an index $i$ such that $\|z_i\|_1$ is
    maximal. Also we may assume that $i<T$. As $\|A\|_1 \leq 1$ it follows that $\| A z_1 \|_1 =
    \|z_i\|_1$ and so $\|z_1\|_1 = \|Az_1\|_1 = \|z_i\|_1$ as $\|z_i\|_1$ was
    maximal. Now it is known for the matrices $A\in {\cal A}$, that
    $\|z_1\|_1 = \|Az_1\|_1, e^\top z_1 = 0$ implies that $Az_1 = z_1$,
    \cite[Lemma~3.8]{wirth2006stochastic}. Finally,
    \begin{align*}
        \left\| A_{T,1}
        z\right\|_T
= \max_{i=1,\ldots,T} \left\| \frac{1}{i} A_1z_1 + \frac{i-1}{i}z_{i-1}  \right\|_1
\leq  \max_{i=1,\ldots,T} \left\{ \frac{c}{i} \left\|z_1 \right\|_1 
     + \frac{i-1}{i} \left\|z_{i-1}  \right\|_1 \right\}  \leq
\max_{i=1,\ldots,T} \left\{ \frac{c}{i} 
     + \frac{i-1}{i} \right\}  \left\|z_T  \right\|_T
 \,.
\end{align*}
This show the assertion. \hfill\qed
\end{proof}

The previous result shows that the iteration of random choices of the
$A_{T,i}$ is contractive when studied with respect to a suitable
norm. This lies the foundation for proving the existence of a unique
invariant and attractive measure for the Markov chain. Before proving
this we need an assumption on the probability functions $\lambda_i$ that
guarantees strong contractivity on average.

\begin{theorem}{\bf (Invariant Measure)}
    \label{t:invmeasure}
    Assume that the probability functions $\lambda_i$ satisfy (A1)-(A3)
    and are Lipschitz continuous. Then for all $T\geq 1$, there exists a
    unique invariant and attractive measure $\pi^T$ on
    $\Sigma^T$. Furthermore, for all $z_0 \in \Sigma^T$, we have that
    almost surely
    \begin{equation}
        \label{eq:ergodic}
         \lim_{k\to \infty} \frac{1}{k} \sum_{\ell =0}^k z(l;z_0) = \int_{\Sigma^T}
  z \ \, \mathrm{d}\pi^T(z) =  \Expect(\pi^T)\,.
    \end{equation}
\end{theorem}

\begin{remark}
    Stronger ergodicity results hold as detailed in
    \cite{barnsley1988invariant,elton1987ergodic}. We skip these for the
    sake of brevity.
\end{remark}

\begin{proof}
    It is easy to show that the sufficient conditions provided in
    \cite{barnsley1988invariant} are satisfied. In particular, these
    conditions can be met by requiring that
    \begin{equation}
    \label{eq:supineq}
        \sup_{z,w\in \Sigma}\ \sum_{A_T \in {\cal A}_T}
 p_A(z) \frac{\|A_{T}(z-w)\|_T}{\| z-w \|_T} < 1 \,.
    \end{equation}
    Note that $z-w \in W$, so Lemma~\ref{lem:ATcont}\,(i) immediately
    implies that the sum does not exceed $1$. Assumption (A3) now ensures
    that for each $z\in \Sigma$, the probability $p_1(z) \geq \lambda_{\min{}}^n
    >0$.  Thus the probability of the matrix $A_1$ is bounded away from
    zero.  Now Lemma~\ref{lem:ATcont}\,(iii) states that $\|A_{T,1}(z-w)\|
    \leq (c+T-1)/T \| z-w \|< \| z-w \|$. As $p_1(z) >0$ for all $z$, we
    see that the supremum in \eqref{eq:supineq} is bounded away from $1$.\\

The final condition that needs to be satisfied is that there exists a
constant $r< 1$ and  a
 constant $\gamma>0$ such that for all $z,w \in \Sigma$ we have
\begin{equation}
    \label{posprobbound}
  \sum_{A \in {\cal A}, \|A_{T}(z-w)\|_T \leq r < 1}    p_A(z)p_A(w) \geq \gamma > 0 \,.
\end{equation}
In our situation, this is clear as $p_1(z) > \lambda_{\min}^n>0$ for all $z
\in \Sigma$. This implies that we may choose $\gamma =
\lambda_{\min}^{2n}$ in \eqref{posprobbound}.

    By Theorem~2.1 in \cite{barnsley1988invariant} the existence of an
    attractive invariant measure follows. Uniqueness is then a consequence
    of attractivity. The ergodic property \eqref{eq:ergodic} now follows
    from \cite{elton1987ergodic}. \hfill\qed
\end{proof}

\begin{remark}
    The previous results shows that the AIMD system is indeed converging
    in a strong sense; in particular, long term averages converge almost
    surely. Simulations suggest, that this limit gets closer to the KKT
    point as $T$ increases. Also for large $T$, with high probability along
    a sample path, the average of the windows of size $T$ is close to the
    KKT point. 
\end{remark}

\subsection{Long-Term Averaging}

We now turn to the situation in which the probabilities for choosing the
matrices depend on the long-term average of the realization.  Note that
\eqref{eq:AIMDfw1} together with \eqref{eq:probdef} do not define a Markov
chain on $\Sigma$, as the probabilities do not depend on the current state
$x(k)$ but rather on the complete history of a sample path.  In order to
obtain a formulation as a Markov chain we include the average in the
state space. 
To this end we introduce the new random variable 
\[
z(k) :=
\begin{bmatrix}
    x(k)^\top & \bar{x}(k)^\top
\end{bmatrix}^\top \,.
\]
It follows from the definition of $\bar{x}(k)$ in \eqref{eq:average} that
$\bar{x}(0) = x_0$ and
\begin{equation}
\bar{x}(k+1) = \frac{1}{k+2} x(k+1) + \frac{k+1}{k+2}\bar{x}(k) \,.
\end{equation}
Hence $z(k)$ evolves according to
\begin{equation}
    \label{eq:AIMD2infav}
    z(k+1) 
    %=: \begin{bmatrix}
        %z_1(k+1) \\ \bar{z}(k+1)
   % \end{bmatrix} = \bar{A}(k) \begin{bmatrix}
       % z_1(k) \\ \bar{z}(k)
    %\end{bmatrix} 
    = \tilde{A}(k) z(k)\,,
\end{equation}
where
\begin{equation}
    \label{eq:Atilde}
    \tilde{A}(k) :=
    \begin{bmatrix}
        A(k)& 0 \\
        ~\\ \frac{1}{k+2} A(k)& \frac{k+1}{k+2} I
    \end{bmatrix} 
    \end{equation}
%If 
Given $A\in {\cal A}$, we introduce 
the matrices
\begin{equation}
    \label{eq:Abar}
    A_{LTA} (k) :=
    \begin{bmatrix}
        A& 0
        ~\\ \\ \frac{1}{k+2} A & \frac{k+1}{k+2} I
    \end{bmatrix}, %\quad i= 1  \ldots,2^n, \quad \text{and} 
    \quad k =0,1, \ldots \,.
\end{equation}
%We then arrive at the Markov chain on $\Sigma \times \Sigma$ defined by
%with the probabilities 
Then, for all $y \in \Sigma$, we have the conditional probabilities 
\begin{equation}
    \label{eq:probdef2}
%    \Prob_{x_0}\bigl(\tilde{A}(k) = A_{LTA} (k)\big|\  \bar{x}(k) =
%    y\bigr) = p_A(y) \,.
\Prob\bigl(\tilde{A}(k) = A_{LTA} (k)\big|\  \bar{x}(k) = y\bigr) = p_A(y) \,.
\end{equation}

This defines a nonhomogeneous Markov chain with place-dependent
probabilities. Note that the nonhomogeneity comes from the time-varying
nature of the matrices $A_{LTA} (k)$, whereas the functions $p_A(\cdot)$
describing the place-dependent probabilities do not depend on time.\\

To obtain contractive properties of the Markov chain \eqref{eq:AIMD2infav}
it will be of interest to study the matrices ${A}_{LTA} (k)$ using a particular
norm. We define a norm on $\R^{2n}$ by setting for $x,y \in \R^n$
\begin{equation}
    \label{eq:normdef}
    \left\| \left[
      \begin{matrix}
          x\\ y
      \end{matrix} \right]
\right\| := \max \{ \|x\|_1 , \|y\|_1 \} \,.
\end{equation}
The matrix norm induced by this vector on $\R^{2n \times 2n}$ is also
denoted by $\|\cdot\|$.

In the following we use the notation ${{\cal A}}_{LTA}  := \{ %\tilde
{A}_{LTA}(k)
\midset %i=1,\ldots,2^n
A \in {\cal A} , k\in\N \}$, which represents the set of all possible
matrices appearing in \eqref{eq:AIMD2infav}.

\begin{lemma}
    \begin{enumerate}[(i)]
      \item For all $%\tilde
{A}_{LTA}  \in %\tilde
{{\cal A}}_{LTA} $
        \begin{equation}
            \label{eq:norm1}
            \|%\tilde
{A}_{LTA}  \| \leq 1\,.
        \end{equation}
      \item The subspace
        \begin{equation}
            \label{eq:Wdef}
            W := \{ ( x,y) \in \R^{2n} \midset e^\top x = e^\top y = 0 \} 
        \end{equation}
        is invariant under all $%\tilde
{A}_{LTA}  \in %\tilde
{{\cal A}}_{LTA}$.
    \end{enumerate}
\end{lemma}

\begin{proof}
    The proof follows the lines of the proof of Lemma~\ref{lem:ATcont} and
    is omitted.
\end{proof}

We stress that the key point of Lemma~\ref{lem:ATcont} was item (iii),
which we used to obtain a uniform contractivity on the state space
$\Sigma^T$ of the Markov chain. A similar result can be obtained in the
present situation, but uniformity is lost due to the time-dependent nature
of the Markov chain. Unfortunately, the constant of contraction converges
to $1$. Considerable effort has been expensed on trying to transfer the
proofs of \cite{barnsley1988invariant,elton1987ergodic} to the present
situation, but to no avail. We thus pursue an entirely different angle of
attack in the proof of our main result.

\begin{theorem}{{\bf (Convergence)}}
    \label{t:convergence}
    Let the functions
    $\lambda_i$ defined in \eqref{eq:lambdadef} satisfy (A1)--(A3) and let
    $x^*$ denote the KKT point guaranteed by Lemma~\ref{lem:Pfix}.
    Consider the nonhomogeneous Markov chain \eqref{eq:AIMD2infav}. For
    any initial condition $z(0)= (x_0,  \bar{x}_0) \in \Sigma^2$ we have that
    the second component of $z(k)$ satisfies
    \begin{equation*}
        \lim_{k\to \infty}\bar{x}(k) = x^\ast \qquad \text{almost
          surely}\quad \Prob_{x_0} \,.
    \end{equation*}
\end{theorem}

\begin{remark}
    Our result says that by local
    modification of the individual probabilities the agents can ensure
    almost sure convergence to the optimum. Inter-agent communication is
    not necessary; rather the only information needed is a $1$-bit intermittent message 
    to all agents that a capacity event has occurred. This minimal information
    suffices for convergence.  
    The main results of this section are the following:
    \begin{enumerate}
      \item an ergodicity result for the algorithm with
        finite-averaging;
      \item a result guaranteeing almost sure convergence to the network
        optimum for the long-term averaging case;
      \item the description of an easily implementable algorithms that
        ensures the convergence of the algorithm to the optimal point,
        using limited uniform
        communication to the agents. The algorithm is particularly suited
        to dumb devices that do not have extensive computational capabilities.
    \end{enumerate}

    Mathematically speaking, it is interesting to see an almost sure
    convergence result, that does not make use of the existence of an
    invariant measure of the stochastic process. We do expect however,
    that when considering the invariant measures $\pi^T$ that are obtained
    for the case of finite time-windows, then as $T\to \infty$ the
    measures $\pi^T$ converge to the Dirac measure in $x^\ast$.
\end{remark}

\section{Related Works}
\label{sec:relworks}
Our work lies in the intersection of two subjects: resource allocation
and limiting characteristics of the stochastic version of the AIMD
algorithm \cite{Jac88}.\newline

The AIMD literature is huge and it is  not straightforward to discuss
the available results in any sort of compact manner here.  We refer
interested readers to some recent works on this topic
in the context of TCP and internet congestion control
\cite{LowPagDoy02,Masaki10,Kumar10,Com06,Mol09,Com06,XuHaRh04,Kel02,Flo03}. Much of this work is based on fluid
approximations of AIMD dynamics; the notable exceptions are
\cite{ShoWirLei06,RotSho08,shorten2007modelling}. The latter of these papers make use
of tools from iterated function systems to deduce the existence of
such a unique probability distribution for standard linear AIMD
networks (of which TCP is an example) under an assumption on the
underlying probability model (albeit under very restrictive
assumptions). To the best of our knowledge, this paper, along with the
companion paper \cite{ShoWirLei06}, established for the first time,
the stochastic convergence of AIMD networks. 
However, the window of infinite length considered in this paper goes well
beyond the set-up in these papers. In particular, the result presented in
this paper may be considered as the limiting case of the results presented
in these papers.\newline

The literature on resource allocation is also immense and a full
review is impossible here. Here, we briefly note that the subject of
resource allocation or social welfare optimization has been studied
in three prominent settings: centralized, distributed concerted, and
distributed competitive.\\

In the first setting, there is a single
decision-maker, who knows the utility function $f_i$ of every agent,
solves the network optimization problem,
and assign the optimal allocation to each agent.\\

In the second setting, every resource user is a decision-maker that
determines its own allocation according to a \emph{fixed} policy (e.g.,
AIMD algorithm in the case of TCP), that is prescribed by a system
operator and remains unchanged over time.
When all agents follow fixed policies prescribed by the system operator, a number of distributed
optimization algorithms have been proposed to iteratively converge to
an optimal allocation of resources for numerous settings
\cite{DAW11,NO09,JRJ09,RNV11}.  Some of these algorithms are based on
achieving consensus \cite{NO09}, others are based on distributed
averaging \cite{ZWSL10,MMMSW09}, and on stochastic approximation
\cite{Bianchi12}. 
All these algorithms rely on communication between agents to achieve
optimality.
For example, when agents are
assigned to nodes in a graph and restricted to communicate only with
neighbors in that graph, the distributed dual averaging algorithm has
been shown to guarantee the convergence of each agent's allocation to
the optimal allocation over iterations of the algorithm \cite{DAW11}.
Our work also considers a distributed setting with fixed-policy
agents, but contains an important difference: the agents do not communicate among
themselves, but are limited to an intermittent feedback signal from the
network. Specifically, they only observe at each iteration whether the
allocation is feasible (i.e., the capacity constraint is satisfied).
Another difference of our work is that our results do not depend on
the existing convergence results from stochastic approximation, and
hence hold under different conditions. In particular, we cannot apply
the standard convergence argument for stochastic approximation because
there are two time-scales (cf. \cite[Chapters~6.2 and
10.4]{Borkar08}). \newline

In a third setting, every resource user is a decision-maker that acts strategically so as to optimize its utility function
with regards to the actions of all agents. When all agents act
strategically, solution concepts such as Nash equilibria are more
meaningful than optimality concepts. In such a setting, a market
mechanism based on bids has been proposed \cite{JohTsi04} and shown to
be efficient in equilibrium under some assumptions
\cite{YuMan06}. In this situation the equilibrium
allocation is also the solution to an optimization problem.  These
works do not however provide a method for the agents to arrive at an
equilibrium allocation. In contrast, our work may not consider
strategic agents, but does present a set of policies that guarantee
the convergence to an optimal allocation.\newline

The link between congestion control (which encompasses the AIMD algorithm) and
optimization has been noted by several authors \cite{low2002understanding,kunniyur2003end,Sri04,Cauld,IJC_Paper,tassiulas2001ratecontrol,Cri2013}. That various embodiments of TCP solve a network utility
maximization problem is a cornerstone of much of the TCP literature
\cite{Sri04}.
However, we are dealing with the converse problem: given an NUM problem,
is there an AIMD algorithm that solves it? Perhaps, the most closely
related works in this direction are given in the following references:
\cite{kunniyur2003end,Cauld,IJC_Paper,tassiulas2001ratecontrol,Cri2013}. Roughly
speaking, these references follow two lines of direction. In the first
direction, fluid-like approximations of congestion control are modified to
address the NUM problem. This yields a sub-gradient like 
algorithm for solving the NUM problem. In the second direction,
synchronized AIMD like algorithms are proposed to solve certain NUM problems
using nonlinear back-off rules and nonlinear increase rules \cite{martin2014,IJC_Paper,tassiulas2001ratecontrol,Cri2013}.
The work presented here goes far beyond these works. First,  we consider the
matrix model of TCP proposed in \cite{ShoWirLei06} as opposed to a fluid
model. 
Fluid approximations are valid only for very large numbers of agents, and the dynamic 
interaction between agents is often overlooked. The matrix model is an exact representation of AIMD dynamics 
under certain assumptions and can readily be implemented in existing software stacks. Furthermore, these models are often analyzed 
using linearized approximations in contrast to our approach in which global stability 
(ergodicity) is proved.  A further difference is that each agent responds
to a capacity event according to its own probability function, known only
to that agent. In this sense our proposed algorithms go beyond traditional AIMD and emulate RED-like congestion control \cite{Sri04}. Second, we assume very limited actuation; an agent only decides to 
respond to a capacity event or not in an asynchronous manner. There is no need for a common clock
and the setting is completely stochastic.  In this context our results prove
convergence and stability of the stochastic AIMD system and establish its
suitability for solving large scale NUM problems.

\section{Example}
\label{sec:example}

We now illustrate the application of our results. To this end consider a total of $n=150$ agents   participating in the optimization. Each agent $i$ has a cost function $f_i$ assigned which maps its share of the resource capacity $C$ to an associated cost. 
The cost functions are chosen from the set of polynomials taking the following forms
\begin{align*}
    g_1(\bar{x}) &= a_1 \bar{x}^2 \\
    g_2(\bar{x}) &=  a_2\bar{x}^{2}  + b_2 \bar{x}^3 \\
    g_3(\bar{x}) &= a_3 \bar{x}^{2}  + b_3 \bar{x}^3  + c_3 x^4\\
   g_4(\bar{x}) &= a_4 \bar{x}^{2}  + b_5 \bar{x}^4  + c_4 x^6
\end{align*}
The parameters $a_j$,$b_j$, and $c_j$ are the cost-factors of each function and are positive. Note that each function is convex and strictly increasing on the interval $[0,C]$. The objective is then 
\begin{equation}
  \label{eq:optimisation-generic}
  \begin{aligned}
& \min_{\bar{x}_{1},\ldots, \bar{x}_{n}} \sum_{i=1}^n f_{i}(\bar{x}_{i}) \\
& \text{s.t.}\quad  \sum_{i=1}^n \bar{x}_i = C.
\end{aligned}
\end{equation}
For the simulations we choose the resource capacity  to be equal to one, i.e. $C=1$. Further, the cost-function type for each agent is selected randomly according to a uniform distribution. 
The cost-factors parameters for each agent are also selected randomly using a uniform distribution between $0$ and $100$. 
Defining
\begin{equation}
\lambda_i(r) = \Gamma \frac{f_i'(r)}{r} \,,
\end{equation}
each agent responds to a capacity event with probability
\begin{eqnarray}\label{eq:prob-set-generic}
\lambda_i(\bar{x}_i(k))
% = \Gamma \frac{f_i'(\bar{x}_i(k))}{\bar{x}_i(k)}
\end{eqnarray}
with $\bar{x}_i(k)$ as in \eqref{eq:average} in case of long-term averaging and 
\begin{eqnarray}\label{eq:prob-set-generic-finite averaging}
\lambda_i(\bar{x}_{i,T}(k))
% = \Gamma \frac{f_i'(\bar{x}_i(k,T))}{\bar{x}_i(k,T)}
\end{eqnarray}
with $\bar{x}_{i,T}(k)$ as in \eqref{eq:finitewindow} in case of finite averaging.
 Here $\Gamma$ is a network wide constant chosen to ensure that
 $0<\lambda_i(r) < 1$ for $0 \le r \le 1$. 
 In our simulations we set $\Gamma = \frac{1}{1300}$. The remaining AIMD parameters are identical for all agents with $\alpha= 0.01$ and $\beta = 0.85$.  From our main result we know that for large $k$ we have that $\bar{x}_i(k) \approx x_i^*$. Thus we can write $\lambda_i(\bar{x}_i(k)) \approx \lambda_i^*$. It follows that 
\begin{eqnarray}
\lambda_i(\bar{x}_i(k)) \approx \Gamma \frac{f_i'(\bar{x}_i(k))}{\Theta}\lambda_i(\bar{x}_i(k))
\end{eqnarray}
and that $f_i'(\bar{x}_i(k)) \approx f_j'(\bar{x}_j(k))$ for all $i,j$ and large $k$, and where $\Theta$ is a network constant. These are precisely the KKT conditions. \protect\footnote{Note that we have made several assumptions. First, we have assumed that eventually $\bar{x}_i(k) \approx x_i^*$. This follows from our main result. We have also assumed that 
$p_i(k) \approx p_i^*$. This follows from continuity of the $f_i(\cdot)$.} 
We first simulate the long-term averaging case, where the average at time instant $k$ is taken over all previous time-steps. Figure~\ref{fig:generic1} shows the typical evolution of the derivative of the cost of seven randomly selected agents. It illustrates that the derivatives approach consensus as $k$ increases;
this results in the     the above stated optimization problem being solved asymptotically.\newline

Figure~\ref{fig:states-generic1} shows the long-term average state
$\bar{x}_i(k)$ of seven randomly chosen agents in comparison to their respective
optimal state $x^*_i$  depicted by a dashed
line.
Figure~\ref{fig:error-generic1} shows the absolute error between the
long-term average and the optimal state for the same seven agents. With
increasing time the long-term average approaches the optimal state for
those seven randomly selected agents. In
Figure~\ref{fig:max-error-generic1} the maximal error between the
long-term average and the optimal state is plotted, which 
approaches zero with increasing time.

\begin{figure}[tp]
  \centering
   \includegraphics[scale=0.9]{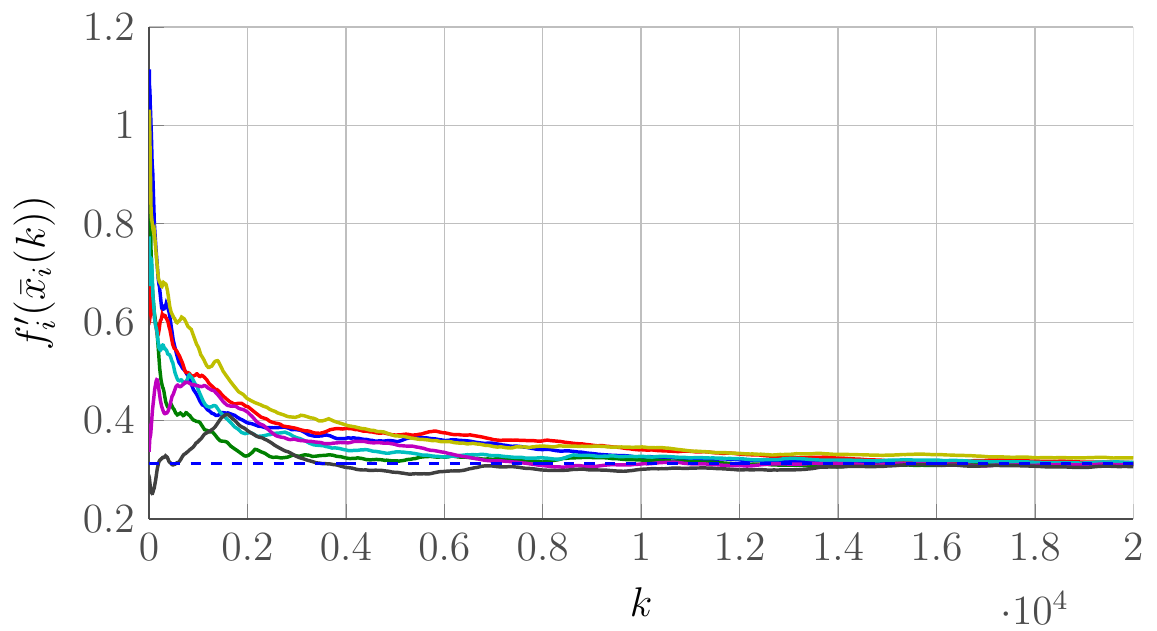}
  \caption{Evolution of the cost derivatives for seven randomly selected agents. 
%  The $x$-axis is a time index.
}
\label{fig:generic1}
\end{figure}

\begin{figure}[tp]
    \centering
    \subfigure[$\bar{x}_{i}(k)$ (solid) and $x_{i}^{*}$ (dashed)]{\includegraphics{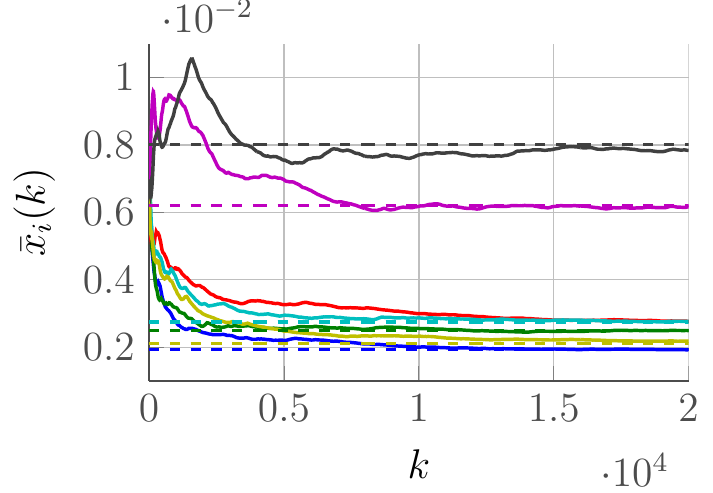}\label{fig:states-generic1}}
    \subfigure[$\left| \bar{x}_{i}(k) - x_{i}^{*} \right|$]{\includegraphics{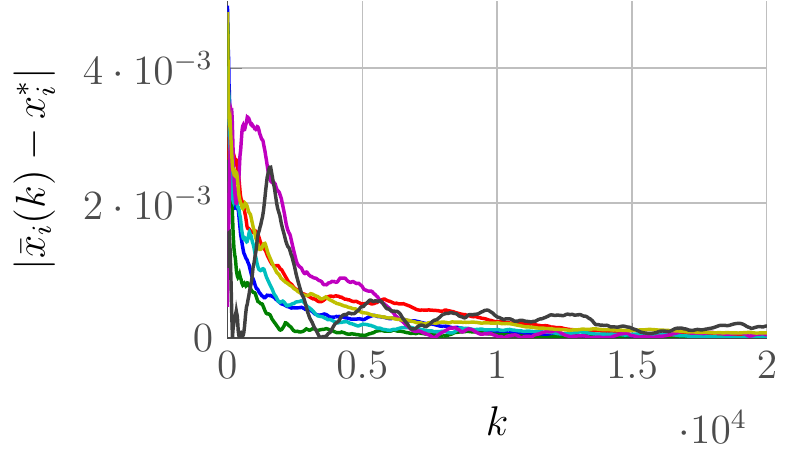}    \label{fig:error-generic1}}
    \caption{Evolution of the states in comparison to the optimal point for seven randomly selected agents.}
    %The $x$-axis is a time index.}
    \label{fig:state-generic1}
\end{figure}

\begin{figure}[tp]
  \centering
   \includegraphics[scale=0.9]{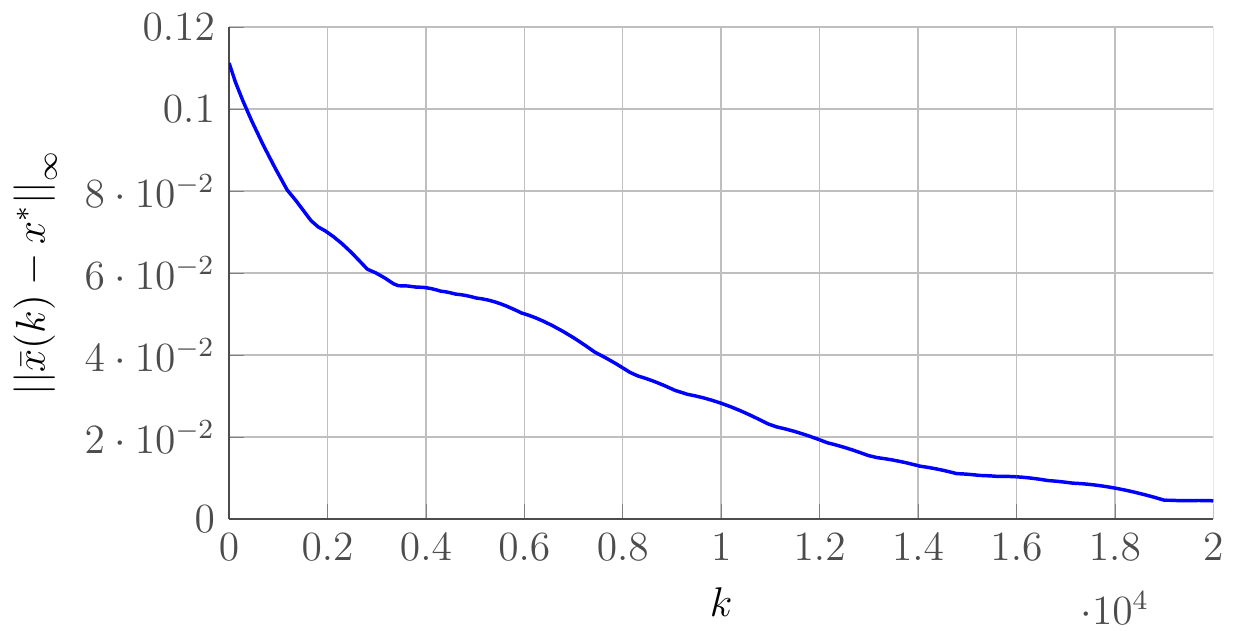}
  \caption{Evolution of the maximal absolute error between the states and the optimal states, i.e. $\left| \left| \bar{x}(k) - x^{*} \right| \right|_{\infty}$. }
  %The $x$-axis is a time index.}
    \label{fig:max-error-generic1}
\end{figure}
~
\newline

We repeat our experiment for the finite averaging case,
\eqref{eq:MPfinitetime}, with a fixed window size $T=
500$. Figure~\ref{fig:generic1-500} shows the typical evolution of the
  derivative of the cost  function for seven randomly selected agents. It illustrates
that these derivatives are    oscillating around the optimal value.
\newline

Figure~\ref{fig:max-error-generic1-500} shows the maximal error between
the long-term average and the optimal state. Recall, that in these
simulations the long term average is not used for determining the drop
probabilities. Its limit exists almost surely and is given by the
expectation of the underlying invariant measure, see
Theorem~\ref{t:invmeasure}. 
While the cost computed with the finite
average is oscillating, the long term average of the state is still
converging towards the optimal value.

\begin{figure}[tp]
  \centering
   \includegraphics[scale=0.9]{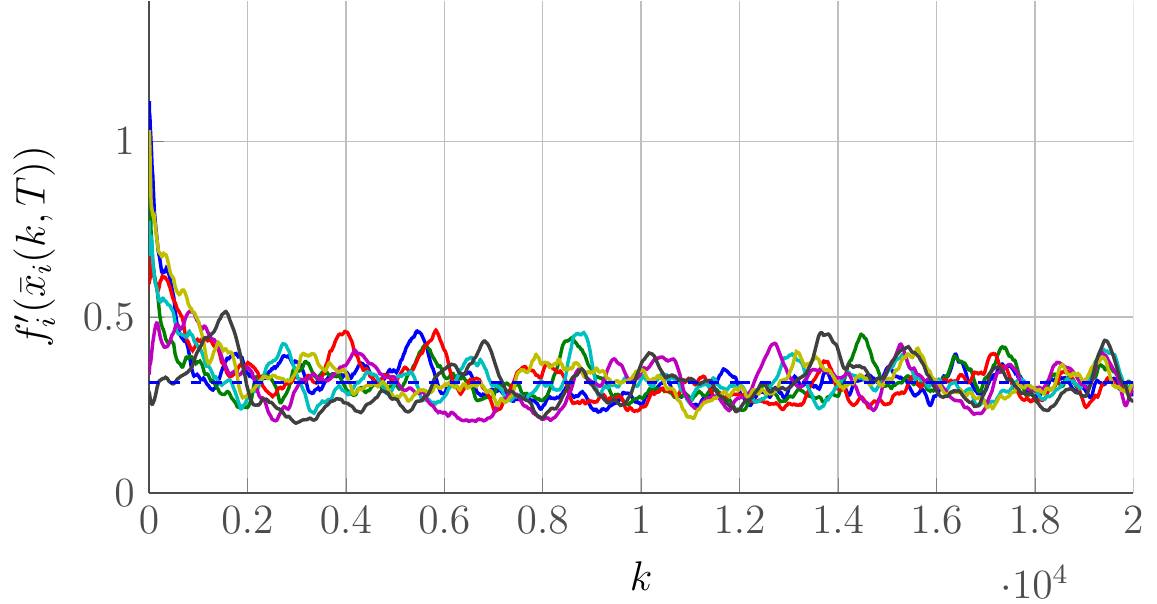}
  \caption{Evolution of the derivatives for seven randomly selected agents. The simulations are done for the finite averaging case with a fixed window size $T=500$.}
 % The $x$-axis is a time index.}
  \label{fig:generic1-500}
\end{figure}

\begin{figure}[tp]
  \centering
   \includegraphics[scale=0.9]{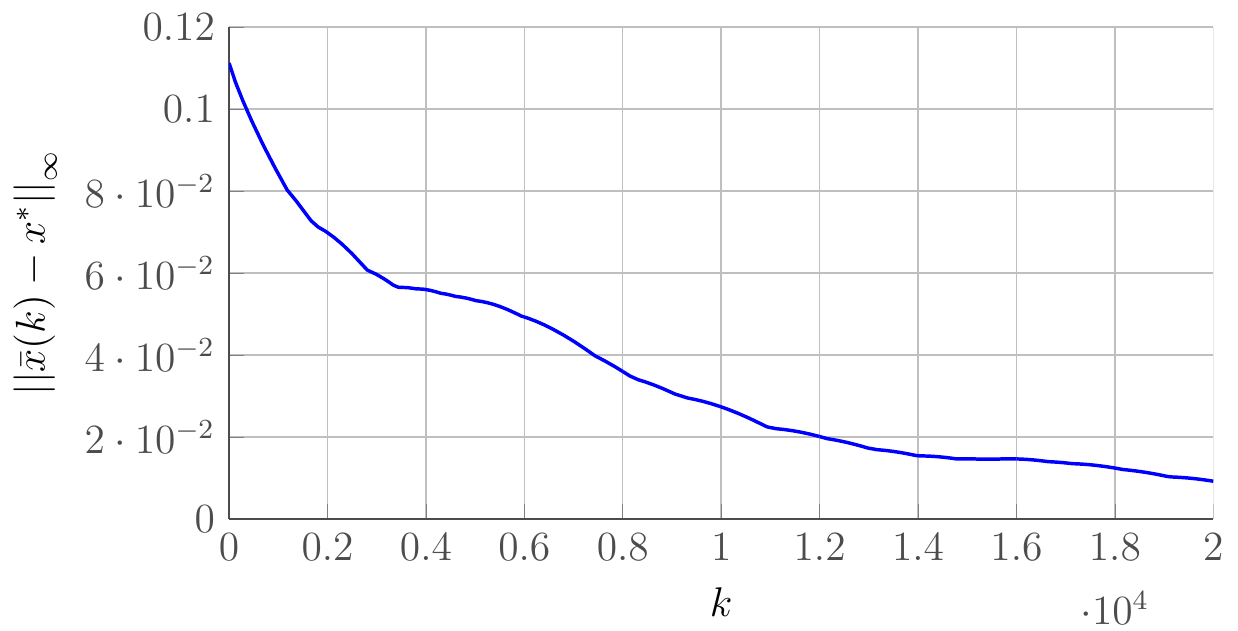}
  \caption{Evolution of the maximal absolute error between the states and the optimal, i.e. $\left| \left| \bar{x}(k) - x^{*} \right| \right|_{\infty}$. The simulations are done for the finite averaging case with a fixed window size $T=500$.}
  %The $x$-axis is a time index.}
  \label{fig:max-error-generic1-500}
\end{figure}

\clearpage
\section{Conclusions}
In this paper we have derived a convergence result for the non-homogeneous
Markov chain that arises in the study of networks employing the {\em
  additive-increase multiplicative decrease} (AIMD) algorithm. We then
used this result to solve the network utility maximization problem
in a very simple manner. Future work will consider the behavior of finite
window averaging systems and elaborate on the preliminary results obtained in this paper.\\
%, and will consider the study of agents that have
%only one of two states - ON and OFF. Initial, but extensive, simulation
%studies suggest that our results hold in this case too.

We have also applied our approach to the distributed optimization of agents that have two possible states---ON and OFF---and finite memory. Initial, but
    extensive, simulation studies suggest that our results hold in this setting as well.
Finally, it is worthwhile to note that it is possible to extend our results in the following directions:
allowing updates to the increase and decrease parameters ($\alpha_i$ and $\beta_i$) over time to achieve faster convergence to optimality, allowing agents leaving and joining the allocation system over time,
and allowing the joint allocation of multiple resources; e.g., bandwidth over distinct links of a communication network, or a combination of bandwidth and computation resources in a cloud server.

\appendix
\section*{Appendix}
\section{Preamble to the Proofs}

We now briefly explain the structure of the proof of the main result, as
otherwise the reader may sometimes wonder why we need certain
intermediate results.\\

The key intuition is that in the long run,  the long term
average $\bar{x}$ changes  slowly.
 In other words, for large $T$ and
relatively short intervals of length $m$ of the form $[T, T+m]$,
$\bar{x}$ is almost a constant, where $m$ is to be understood to be small
when compared to $T$. The reason for this is the simple relation
\begin{equation}
\label{eq:explaindyn}
   \bar{x}(T+m) = \frac{T+1}{T+1+m} \bar{x}(T) + \frac{m}{T+1+m}\left( \frac{1}{m}\sum_{\ell=1}^m x(T+\ell) \right)\,,
\end{equation}
which holds along any sample path.\\

If $\bar{x}$ is almost a constant on a certain interval, then the
probabilities for choosing the matrices $A\in {\cal A}$ are almost
constant, and we can approximate the dynamics using the results on AIMD
with constant probabilities; and consequently Lemma~\ref{lem:uniformfixed}
becomes relevant. This result says that, provided that $m$ is large
enough, the average over the next $m$ steps is close to the expectation of
the AIMD Markov chain with constant probabilities. And this holds for all
starting
conditions $x(T)$ and with high probability.\\

While this basic intuition turns out to be true, we need to resolve the
fact that the ergodic limit of the ``fixed-probability system'' depends on
$T$. Specifically, $m$ and $T$ depend on each other and the precise
resolution of our proof depends on understanding this relationship.\\

To resolve this, we use the following interpretation of
\eqref{eq:explaindyn}. 
For $y \in \Sigma$, we denote by $P(y)$ the expectation of the
invariant measure of the IID AIMD process with  fixed probabilities $\lambda_1(y_1) \dots, \lambda_n(y_n)$,
that is, 
\begin{equation}
P(y) = \xi_{\lambda(y)} \,,
\end{equation}
where
\[
\lambda(y) =
\left[ \begin{array}{ccc}
\lambda_1(y_1) &\cdots & \lambda_n(y_n)
\end{array}
\right] \,.
\] 
We then rewrite
\eqref{eq:explaindyn} as
\begin{equation}
\label{eq:explaindyn2}
   \bar{x}(T+m) = \frac{T+1}{T+1+m} \bar{x}(T) + \frac{m}{T+1+m}\left( 
P(\bar{x}(T)) + \Delta(T) \right)\,, 
\end{equation}
where we interpret $\Delta(T)$ as a suitable perturbation term, that
aggregates the effect that the probabilities are not precisely constant on
$[T+1,T+m]$, and the further effect that we are not at the expectation but
only close to it.\\

To understand the dynamics in \eqref{eq:explaindyn2} we study the system
\begin{equation}
\label{eq:explaindyn3}
   \bar{x}(T+m) = \frac{T+1}{T+1+m} \bar{x}(T) + \frac{m}{T+1+m}\,
P(\bar{x}(T)) \,, 
\end{equation}
and interpret system \eqref{eq:explaindyn2} as a perturbed version
thereof. This is the sole purpose of Appendix~\ref{sec:deterministic}, in
which we obtain (i) characterizations of the unique fixed point of
\eqref{eq:explaindyn3}, (ii) characterizations of attractivity properties
of neighborhoods of this fixed point in dependence of the size of
$m/(T+m)$, and (iii) the necessary robustness results to extend these
attractivity statements to the perturbed system~\eqref{eq:explaindyn2}.  In
Appendix~\ref{sec:mainproof} we then bring the stochastic nature of our
nonhomogeneous Markov chain into play and use the results of
Appendix~\ref{sec:deterministic} to prove almost sure convergence using
what is essentially a Lyapunov type argument.

\section{Deterministic Iteration}
\label{sec:deterministic}

In this section we present a collection of stability and robustness
results for a deterministic system closely related to the AIMD Markov
chain. These results will turn out to be instrumental in the proof of
the main result Theorem~\ref{t:convergence}.\\

The first results study a deterministic system defined by
successive convex combinations of a point in $\Sigma$ with the expectation
of this point as defined through \eqref{eq:expfixedprob}.\newline

 Recall, that we assume that $\alpha \in \ri \Sigma$ and 
$\beta \in (0,1)^n$ satisfy the assumption that the quotient
$\alpha_i/(1-\beta_i)$ is a constant independent of $i$. 
As a consequence
the limiting value defined in \eqref{eq:expfixedprob} simplify. Given the
%probabilities $\lambda_i(\cdot), i=1,\ldots,n$ at a given point in $\Sigma$, 
probabilities $\lambda_1, \dots, \lambda_n$ the expression reduces to
 \begin{equation}
      \label{eq:expfixedprob2}  
\xi_\lambda =\frac{1}{\sum_{\ell=1}^n \lambda_\ell^{-1}}
    \begin{bmatrix}
        \lambda_1^{-1} \\ \vdots \\
\lambda_n^{-1}
    \end{bmatrix}\,.
    \end{equation}
We thus arrive at the map  $P : \Sigma \to \Sigma$ given by
\begin{equation}
    \label{eq:Ex}
P(x) = 
%\frac{1}{\sum_{l=1}^n \lambda_\ell(x_\ell)^{-1}}
\Theta(x)
    \begin{bmatrix}
        \lambda_1(x_1)^{-1} \\ \vdots \\
\lambda_n(x_n)^{-1}
    \end{bmatrix}
    \qquad \text{with} \qquad
   \Theta(x) =  \frac{1}{\sum_{\ell =1}^n \lambda_\ell(x_\ell)^{-1}} \,.
\end{equation}

Note that $P(\Sigma) \subset \ri \Sigma$ is compact by (A3). We may therefore choose a
constant $\delta^{-{}}>0$ such that 
\begin{equation}
\label{eq:deltaminusdef}
    P(\Sigma) + \overline{B}_1(0,\delta^{-{}}) \subset \conv P(\Sigma) + \overline{B}_1(0,2\delta^{-{}}) \subset \ri \Sigma\,.
\end{equation}
Note that in this instance, and in the following, scalings and sums of
sets are in the standard sense of Minkowski sums. Also the factor $2$ is
an arbitrarily chosen factor that will become useful in later robustness
estimates. All that is required is that this factor exceeds $1$.
Furthermore, we require the constant
\begin{equation}
    \label{eq:deltaplus}
    \delta^+ := \max \{ \distone(y,P(\Sigma)) \;;\; y \in \Sigma \ \} \,.
\end{equation}

We will be interested in systems that perform successive convex
combinations of the state $x$ and $P(x)$.
For $\{ \varepsilon_k \}_{k\in\N} \subset (0,1)$ consider the system\footnote{This system is an instance of stochastic
  approximation, and convergence results of
  \cite[Chapter 2]{Borkar08} hold under appropriate assumptions. However, we shall require and derive stronger results.}
\begin{equation}
    \label{eq:convexcombsys}
    x(k+1) = (1-\varepsilon_k) x(k) + \varepsilon_k P(x(k))\,.
\end{equation}

We note the following simple properties of the iteration in
\eqref{eq:convexcombsys}.

\begin{lemma}
\label{lem:Pfix}
Suppose that Assumptions (A1)--(A3) hold and $n\geq2$.
Let $\alpha \in \ri \Sigma_n$ and 
$\beta \in (0,1)^n$ be  such that $\alpha_i/(1-\beta_i)$ is independent of
$i$. 
Then $P$ has the following properties.
    \begin{enumerate}[(i)]
      \item  $P$ has a fixed point, that is, there is  a vector $x^\ast \in \Sigma$ with
        $P(x^\ast) = x^\ast$.
      \item The fixed point $x^\ast$ is unique and  is characterized by the property
        \begin{equation}
            \label{eq:Cequality}
            x_i^\ast\lambda_i(x_i^\ast)=x_j^\ast\lambda_j(x_j^\ast) := \gamma_F
            % =: \gamma(x^*) \,,
            \qquad \text{for all} \quad i,j \in \{ 1,\ldots,n \} \,.
        \end{equation}
      \item 
      For every $\varepsilon\in (0,1]$ the fixed point $x^\ast$ of $P$  
        is the unique fixed point of
        \begin{equation*}
            x \mapsto (1-\varepsilon) x + \varepsilon P(x)\,.
        \end{equation*}
      \item For every $x_0 \in \Sigma$ and every sequence $\{
        \varepsilon_k \}_{k\in\N} \subset (0,1)$ the solution of
        \eqref{eq:convexcombsys} satisfies $x(k) \gg 0$ for all $k\geq 1$.
    \end{enumerate}
\end{lemma}

\begin{proof}
    (i)~As $P: \Sigma\to \Sigma$ is continuous  and $\Sigma$ is compact and convex,
    the existence of a fixed
    point for  $P$ follows from Brouwer's fixed point theorem. 
    
    (ii)~ 
    Letting
     $\gamma(x^*): = \sum_{\ell=1}^n 1/\lambda_\ell(x^*_\ell)$,
     it follows from the definition of $P$ that a fixed point  $x^*$ is characterized by  
    \[
   \frac{ \gamma(x^*)}{ \lambda_i(x^*_i)} = x^*_i
    \]
    for $i=1, \ldots, n$, that is,
    \[
    \lambda_i(x^*_i)x^*_i =\gamma(x^*)
    \]
    Suppose that  there are
    two fixed points $x^\ast \neq y^\ast$ for $P$.
    Since  $x^\ast, y^\ast \in \Sigma$, 
    there are indices $i$ and $j$ such that
    $x^\ast_i >
    y^\ast_i$ and $x^\ast_j < y^\ast_j$.  
Also,
    \begin{equation*}
       x^\ast_i \lambda_i(x^\ast_i)  =  x^\ast_j \lambda_j(x_j^\ast)
        \quad \text{and} \quad   y^\ast_i \lambda_i(y^\ast_i)   = y^\ast_j \lambda_j(y_j^\ast)\,.
    \end{equation*}
But from Assumption~(A2) 
we have $ x^\ast_i \lambda_i(x_i^\ast) > y^\ast_i \lambda_i(y_i^\ast)
= y^\ast_j \lambda_j(y_j^\ast) > x^\ast_j \lambda_j(x_j^\ast)$.
This contradiction completes the proof.

(iii)~This is an immediate consequence of (i).

(iv)~This follows as $P(x) \gg 0$ for all $x \in \Sigma$ by definition and
using Assumption~(A3). \hfill\qed
\end{proof}

In order to simplify notation, we introduce for $\varepsilon\in [0,1]$ the
map $R_\varepsilon:\Sigma\to \Sigma$ by
\begin{equation}
    \label{eq:Repsdef}
    R_\varepsilon(x) := (1-\varepsilon) x + \varepsilon P(x)\,.
\end{equation}
Lemma \ref{lem:Pfix} tells us that  for $\varepsilon\in (0,1]$ the fixed point $x^\ast$
of $P$ is also the unique fixed point of $R_\varepsilon$.\\

In our analysis of the dynamics we require two types of contractive
properties of the map $R_\varepsilon$ in combination with robustness
results. We will also consider set-valued maps of the form
\begin{equation}
    \label{eq:Repsdefincl}
    \Psi_\varepsilon^\delta (x) := R_\varepsilon(x) + \varepsilon \overline{B}_1(0,\delta) = (1-\varepsilon) x + \varepsilon (P(x) + \overline{B}_1(0,\delta))\,,
\end{equation}
where we assume $0< \delta < \delta^{-{}}$. Note that by definition of
$\delta^{-{}}$ this ensures that $ \Psi_\varepsilon^\delta (x) \subset\ri \Sigma $.
In the following lemma, we analyze properties of the map
$\Psi_\varepsilon^\delta$ by studying individual elements in its image.

    The next result describes two important features of the iteration
    \begin{equation}
        \label{eq:incl}
        x(k+1) \in \Psi_\varepsilon^\delta(x(k))\,.
    \end{equation}
    On one hand by (i) the iteration converges with rate $(1-\varepsilon)$
    to the convex set
\[
P_{\mathrm{co}}( \delta):= \conv
      (P(\Sigma) + \overline{B}_1(0, \delta))
      \]
      On the
    other hand using (ii) if the iteration is perturbed so that all we
    know that there is a convex combination with some $y\in \Sigma$ then
    we may bound the increase of the distance to the convex set.  Finally,
    by (iii) the error induced by the perturbation $y$ can be linearly
    bounded in $\varepsilon$, provided that we are sufficiently far way
    from $P_{\mathrm{co}}(\delta)$. For the following
    statement recall the definition of $\delta^-$ in \eqref{eq:deltaminusdef}.

\begin{lemma}
    \label{lem:convinv}
    Let $x\in \Sigma$. Then for all $0<\varepsilon\leq 1$:
    \begin{enumerate}[(i)]
      \item For all $0<\delta< \delta^{-{}}$ and $\Delta\in \R^n,
        e^\top\Delta=0, \|\Delta\|_1 \leq \delta$ we have
\begin{equation}
        \label{eq:convinv}
        \distone(R_\varepsilon(x) + \varepsilon \Delta, \, P_{\mathrm{co}}(\delta)) \leq (1-\varepsilon) 
        \distone(x, P_{\mathrm{co}}(\delta))\,.
    \end{equation}

  \item In view of \eqref{eq:deltaplus}, for all $0<\delta< \delta^{-{}}$
    and all $y\in \Sigma$, we have
\begin{equation}
\label{eq:deltaplusbound}
   \distone((1-\varepsilon)x + \varepsilon y, 
   P_{\mathrm{co}}(\delta)) \leq  (1-\varepsilon) \distone(x, P_{\mathrm{co}}(\delta)) + \varepsilon \delta^{+}\,.
\end{equation}
\item For every $0< \bar{\delta} < \delta^{-}$ there exists a
$C_{\bar \delta}>0$ such that for all $0< \varepsilon < 1$ and all $0<
\delta < \delta^-$ we have the following implication: If $x\in \Sigma$
satisfies $\distone(x, P_{\mathrm{co}}(\delta)) >
\bar{\delta}$ and $y\in \Sigma$, 
then
\begin{equation}
    \label{eq:contr1}
    \distone((1-\varepsilon)x + \varepsilon y,  P_{\mathrm{co}}(\delta) ) \leq  (1+ C_{\bar \delta} \varepsilon) \distone(x, P_{\mathrm{co}}(\delta))\,.
\end{equation}
    \end{enumerate}
\end{lemma}

\begin{proof}
  (i)~  Let $z \in 
  P_{\mathrm{co}}(\delta)$ be such that 
        \begin{equation*}
            \| x-z \|_1 =  \distone(x, P_{\mathrm{co}}(\delta))\,.
        \end{equation*}
Then by convexity $(1-\varepsilon) z + \varepsilon (P(x) +\Delta) \in P_{\mathrm{co}}(\delta)$ and so
        \begin{equation*}
\distone(R_\varepsilon(x) + \varepsilon \Delta, 
P_{\mathrm{co}}(\delta)) \leq
\|  ( R_\varepsilon(x) + \varepsilon \Delta) - ((1-\varepsilon) z + \varepsilon (P(x) +\Delta)) \|_1
= (1-\varepsilon) \| x-z \|_1\,.
        \end{equation*}
(ii)~To prove \eqref{eq:deltaplusbound} note that for any convex set $C$, 
we have $C= (1-\varepsilon)C + \varepsilon C$. 
Hence,
\[
    \distone((1-\varepsilon)x + \varepsilon y, 
    P_{\mathrm{co}}(\delta) ) \leq (1-\varepsilon) 
    \distone(x, P_{\mathrm{co}}(\delta) ) + \varepsilon \ 
     \distone(y, P_{\mathrm{co}}(\delta) )\,,
\]
which shows the claim by definition of $\delta^+$.\\

(iii)~To prove \eqref{eq:contr1} note that with the assumption
$\distone(x, P_{\mathrm{co}}(\delta)) >
 \bar{\delta}$ we arrive at
\begin{equation*}
    \delta^{+} \leq \delta^{+}\frac{\distone(x, P_{\mathrm{co}}(\delta) )}{\bar{\delta}}, 
\end{equation*}
and so \eqref{eq:contr1} follows from \eqref{eq:deltaplusbound} with an
appropriate choice of $C_{\bar{\delta}}>0$.  This completes the
proof. \hfill\qed
\end{proof}

It is the aim of the following sequence of results to establish similar
properties close to the fixed point $x^\ast$. To this end we have found it
necessary to work with a different metric.  

We need the following lemma, for which we will make use of
the following elementary observations. First, note the implication
    \begin{equation}
        \label{eq:obs1}
        \bigl( x \in \ri \Sigma \quad\text{ and }\quad x\neq x^\ast \bigr)
        \quad \Rightarrow \quad 
\min_j \left\{x_j / x_j^\ast \right \} < 1 <\max_i \left\{x_i / x_i^\ast \right \}\,.
    \end{equation}
    Using this relation it is straightforward to see that for any sequence
    $\{ x_k \} \subset \ri \Sigma$ we have the equivalence
\begin{equation}
    \label{eq:obs2}
    e^{d_H}(x_k,x^\ast) \to 1 \quad \Leftrightarrow \quad \min_j
    \{ {x_{kj}}/{x_j^\ast} \}  \to 1 \quad \Leftrightarrow \quad
\max_i \{ 
    {x_{ki}}/{x_i^\ast} \}  \to 1 \,.
\end{equation}

\begin{lemma}
    \label{lem:Hilbertfar}
    Let $x^\ast \in \Sigma$ be the unique fixed point of $P$, as described
    in Lemma~\ref{lem:Pfix}. For every $\eta > 0$, there are constants
    $0<r<1<R$ and a constant 
    $\varepsilon_0 \in (0,1)$ such that for all $0 < \varepsilon <
    \varepsilon_0$ we have that if $d_H(x,x^\ast)> \eta$ and $i,j$ are such that
 \begin{equation}
\label{eq:REHilbert}
     e^{d_H}(R_\varepsilon(x),x^\ast)=   \frac{ R_\varepsilon(x)_i / x_i^\ast }{ R_\varepsilon(x)_j/ x_j^\ast} \,,
    \end{equation}
then $x_i > R x_i^\ast$ and $x_j < r x_j^\ast$. 
\end{lemma}

\begin{proof}
    Let $\eta>0$ be fixed. 
    %Assume that the value in equality \eqref{eq:REHilbert} is attained in a pair $(i,j)$
    %such that $x_i \leq x_i^\ast$ or $x_j \geq x_j^\ast$. It is thus sufficient to show that we may
    %choose $\varepsilon_0>0$ small enough so that this can only occur for $x
    %\in B_H(x^\ast, \eta)$.  
    %We will make use
    %of three basic inequalities which we derive now.  
    By \eqref{eq:obs2}
    there exists a constant $r_1 \in (0,1)$ such that for all $x\in \ri
    \Sigma$ with $d_H(x,x^\ast)\geq \eta$ we have
    \begin{equation}
        \label{eq:minetabound}
        \min_{j=1,\ldots,n} \ \frac{x_j}{x_j^\ast} \leq r_1 < 1\,.
    \end{equation}
    Using (A3) and $\lambda_i(x_i) \in [\lambda_{\min{}},1]$, we
    have for all $x\in \Sigma$, $i=1,\ldots,n$ that
    \begin{equation}
        \label{eq:Pibound}
       0< c_1 := \frac{\lambda_{\min}}{n}   \leq P_i(x) = \frac{1}{\sum_{\ell=1}^n \lambda_\ell(x_\ell)^{-1}}\ \frac{1}{\lambda_i(x_i)} \leq \min\left\{1,\frac{1}{n}\frac{1}{\lambda_{\min}} \right\} =: c_2\,.
    \end{equation}
    Define $x^\ast_{\min}:= \min_{i=1,\ldots,n} x_i^\ast >0 , x^\ast_{\max}:=
    \max_{i=1,\ldots,n} x_i^\ast>0$.
% \begin{equation}
%     \label{eq:epsbound1}
%     \varepsilon_1 := \frac{1-r}{1-r + \frac{c_2}{x^\ast_{\min}} - \frac{c_1}{x^\ast_{\max}}}\,.
% \end{equation}
Choose $r,r_2$ such that $0 < r_1 < r_2 < r <1$ and $\varepsilon_1>0$ such
that we have for all $\varepsilon \in (0,\varepsilon_1)$ that 
\begin{equation}
    \label{eq:2}
    (1-\varepsilon) r_1 + \varepsilon \frac{c_2}{x^\ast_{\min}} < r_2 <
    (1-\varepsilon) r + \varepsilon \frac{c_1}{x^\ast_{\max}} \,.
\end{equation}
With this choice it follows that if $x_j/x^\ast_j \leq r_1$ and $0< \varepsilon <
\varepsilon_1$ then
\begin{equation}
    \label{eq:3}
    \frac{R_\varepsilon(x)_j}{x_j^\ast} = \frac{(1-\varepsilon)x_j + \varepsilon P_j(x)}{x_j^\ast}
        \leq (1-\varepsilon) r_1 + \varepsilon \frac{c_2}{x^\ast_{\min}} <r_2 \,.
\end{equation}
On the other hand, if $x_j/x^\ast_j \geq r$ and $0< \varepsilon <
\varepsilon_1$ then 
\begin{equation}
    \label{eq:4}
   \frac{R_\varepsilon(x)_j}{x_j^\ast} = \frac{(1-\varepsilon)x_j + \varepsilon P_j(x)}{x_j^\ast}
        \geq (1-\varepsilon) r + \varepsilon \frac{c_1}{x^\ast_{\max}} >r_2 \,. 
\end{equation}
Combining \eqref{eq:3} and \eqref{eq:4} we see that if $d_H(x,x^\ast)\geq
\eta$ and $r, \varepsilon_1$ are chosen as above then \eqref{eq:REHilbert} implies
that $x_j < r x_j^\ast$, as desired.

% We claim that the assumptions that $d_H(x,x^\ast)\geq \eta$, and
% $0<\varepsilon< \varepsilon_1$ imply that
% \begin{equation}
%     \label{eq:result1}
%     \frac{R_\varepsilon(x)_j}{x_j^\ast} = \min_{\ell=1,\ldots,n} \frac{R_\varepsilon(x)_\ell}{x_\ell^\ast}
% \quad \Rightarrow \quad x_j < x_j^\ast \,.
% \end{equation}

% So fix $x\in \ri \Sigma, d_H(x,x^\ast)\geq \eta$ and an index $j$ so that
% $x_j/x_j^\ast \leq r$, where we have used \eqref{eq:minetabound}. Let $\ell\in \{
% 1,\ldots,n \}$ be any index for which $x_\ell \geq x_\ell^\ast$. Then we
% obtain using \eqref{eq:minetabound}, \eqref{eq:maxetabound},
% \eqref{eq:Pibound} as well as the definition of $x^\ast_{\min},
% x^\ast_{\max}$ that
%     \begin{equation*}
%         \frac{R_\varepsilon(x)_j}{x_j^\ast} =   \frac{(1-\varepsilon)x_j + \varepsilon P_j(x)}{x_j^\ast}
%         \leq (1-\varepsilon) r + \varepsilon \frac{c_2}{x^\ast_{\min}}
%         \stackrel{\eqref{eq:epsbound1}}{<} (1-\varepsilon) + \varepsilon \frac{c_1}{x^\ast_{\max}}
%         \leq \frac{(1-\varepsilon)x_\ell + \varepsilon P_\ell(x)}{x_\ell^\ast} \,.
%     \end{equation*}
%     This shows \eqref{eq:result1}. 

The claim for the upper bound $x_i \geq R x_i^\ast$ follows with a similar
argument. To this end note that by \eqref{eq:obs2}, there exists a constant $R_1>1$ such
    that for all $x\in \ri \Sigma$ with $d_H(x,x^\ast)\geq \eta$ we have
    \begin{equation}
        \label{eq:maxetabound}
        \max_{i=1,\ldots,n} \ \frac{x_i}{x_i^\ast} \geq R_1 > 1\,.
    \end{equation}
The claim then follows by another application of \eqref{eq:Pibound}.
% With a similar calculation we obtain
%     that with
% \begin{equation}
%     \label{eq:epsbound2}
%     \varepsilon_2 := \frac{R-1}{R-1 + \frac{c_2}{x^\ast_{\min}}
% - \frac{c_1}{x^\ast_{\max}}} \,, 
% \end{equation}
% the conditions $d_H(x,x^\ast)\geq \eta$, and
% $0<\varepsilon< \varepsilon_2$ imply that
% \begin{equation}
%     \label{eq:result2}
%     \frac{R_\varepsilon(x)_i}{x_i^\ast} = \max_{\ell=1,\ldots,n} \frac{R_\varepsilon(x)_\ell}{x_\ell^\ast}
% \quad \Rightarrow \quad x_i > x_i^\ast \,.
% \end{equation}
%
% Summarizing, we obtain the following result from \eqref{eq:result1} and
% \eqref{eq:result2}: If $0<\varepsilon < \varepsilon_0 := \min \{
% \varepsilon_1,\varepsilon_2 \}$, then $d_H(x,x^\ast) \geq \eta$ implies
% that the premise of \eqref{eq:estimate3} is satisfied for the pair $(i,j)$
% of indices such that
% $(R_\varepsilon(x_i)/x_i^\ast)/(R_\varepsilon(x_j)/x_j^\ast) =
% e^{d_H}(R_\varepsilon(x),x^\ast)$. 
\end{proof} 

The following result is a
cornerstone in our proof of the main result.

\begin{theorem}
\label{lem:basicconv}
    Let $x^\ast \in \Sigma$ be the unique fixed point of $P$, as described
    in Lemma~\ref{lem:Pfix}. For every $\eta > 0$, there is
    $1>\varepsilon_0>0$ such that for all $0 < \varepsilon < \varepsilon_0$
    we have
    \begin{equation}
        \label{eq:basiccontr}
        d_H(x,x^\ast) > \eta \quad \Rightarrow \quad  
d_H(R_\varepsilon(x),x^\ast) < d_H(x,x^\ast) \,.
    \end{equation}
\end{theorem}

\begin{proof}
    If $x \in \Sigma$ and some entries of $x$ are zero, then
    $d_H(x,x^\ast) = \infty$, and also $R_\varepsilon(x) \gg 0$ by
    construction. Thus the claim follows trivially.
    In the remainder of the proof we will thus assume that $x \gg 0$. 

%     Recall that the definition yields
%     \begin{equation}
% \label{eq:REHilbert}
%      e^{d_H}(R_\varepsilon(x),x^\ast)=   \frac{ \max_i \left\{R_\varepsilon(x)_i / x_i^\ast \right \}}{\min_j \left\{ R_\varepsilon(x)_j/ x_j^\ast \right \}} \,.
%     \end{equation}
Fix $\eta> 0$. By Lemma \ref{lem:Hilbertfar} we may choose a constant
$\varepsilon_1\in (0,1)$ such that if $d_H(x,x^\ast) > \eta$ and if $i,j$ are such that  
\begin{equation}
\label{eq:REHilbert_bis}
     e^{d_H}(R_\varepsilon(x),x^\ast)=   \frac{ R_\varepsilon(x)_i / x_i^\ast }{ R_\varepsilon(x)_j/ x_j^\ast} \,,
    \end{equation}
then $x_i > x_i^\ast$ and $x_j < x_j^\ast$.

    For the case $x_i > x_i^\ast$, we obtain using the constant $\gamma_F$
    defined in \eqref{eq:Cequality} and the fact that $x_i \lambda_i(x_i)
    >\gamma_F$ by assumption (A2) that
    \begin{equation}
\label{eq:estimate1}
        \begin{aligned}
        R_\varepsilon(x)_i &= 
  (1-\varepsilon)x_i +  \frac{\varepsilon}{\sum_{\nu=1}^n \lambda_\nu(x_\nu)^{-1}}\frac{1}{\lambda_i(x_i)}\\ 
&= \left( (1-\varepsilon) + \frac{\varepsilon}{\sum_{\nu=1}^n 
\lambda_\nu(x_\nu)^{-1}} \frac{1}{x_i \lambda_i(x_i)} \right) x_i
<  \left( (1-\varepsilon) + \frac{\varepsilon}{\sum_{\nu=1}^n 
\lambda_\nu(x_\nu)^{-1}} \frac{1}{\gamma_F} \right) x_i       \,.     
        \end{aligned}
    \end{equation}
    By a similar argument, if $x_j < x_j^\ast$ we obtain 
\begin{equation}
    \label{eq:estimate2}
        R_\varepsilon(x)_j = (1-\varepsilon)x_j + \frac{\varepsilon}{\sum_{\nu=1}^n \lambda_\nu(x_\nu)^{-1}} \frac{1}{\lambda_j(x_j)}
>  \left( (1-\varepsilon) + \frac{\varepsilon}{\sum_{\nu=1}^n \lambda_\nu(x_\nu)^{-1}} \frac{1}{\gamma_F} \right) x_j      \,.  
\end{equation}
Combining \eqref{eq:estimate1} and \eqref{eq:estimate2}, we obtain for the
indices $i,j$ such that \eqref{eq:REHilbert_bis} holds that
    \begin{equation}
        \label{eq:estimate3}
        e^{d_H}(R_\varepsilon(x),x^\ast)= \frac{R_\varepsilon(x)_i / x_i^\ast}{R_\varepsilon(x)_j / x_j^\ast}
        < \frac{ x_i / x_i^\ast}{x_j / x_j^\ast} \leq e^{d_H}(x,x^\ast)\,.
    \end{equation}
This completes the proof.
\end{proof}

\begin{corollary}
    \label{c:Pconvestimate}
     Let $x^\ast \in \Sigma$ be the unique fixed point of $P$, as
     described in Lemma~\ref{lem:Pfix}. For every $\eta > 0$, there is
     $1>\varepsilon_0>0$ and a constant $C_\eta>0$ such that for all $0 <
     \varepsilon < \varepsilon_0$ we have
    \begin{equation}
        \label{eq:basiccontr2}
        d_H(x,x^\ast) \geq \eta \quad \Rightarrow \quad  
e^{d_H}(R_\varepsilon(x),x^\ast) < (1 - C_\eta \varepsilon) e^{d_H}(x,x^\ast) \,,
    \end{equation}
or equivalently,
\begin{equation}
    \label{eq:basicconstr3}
         d_H(x,x^\ast) \geq \eta \quad \Rightarrow \quad  
d_H(R_\varepsilon(x),x^\ast) < d_H(x,x^\ast) + \log(1-C_\eta \varepsilon)
< d_H(x,x^\ast) - C_\eta \varepsilon\,.   
\end{equation}
\end{corollary}

\begin{proof}
    Fix $\eta>0$. Let $\varepsilon_0>0$ and $r,R$ be the constants
    corresponding to $\eta$ given by Lemma \ref{lem:Hilbertfar}. Using Assumptions~(A1), (A2) and
    Lemma~\ref{lem:Pfix}, there are constant $L_1 < \gamma_F < L_2$ such that 
    \begin{align*}
        x_j \lambda_j(x_j) \leq L_1 &\quad \text{ for all } j=1,\ldots,n, \text{ and all }  x_j\in [0,1] \text{ such that }  \frac{x_j}{x_j^\ast} \leq r, \\
        x_i \lambda_i(x_i) \geq L_2 &\quad \text{ for all } i=1,\ldots,n,
        \text{ and all } x_i\in [0,1] \text{ such that } \frac{x_i}{x_i^\ast} \geq R \,.
    \end{align*}
    With this notation we can refine the inequalities \eqref{eq:estimate1} and
    \eqref{eq:estimate2}. Namely, if $d_H(x,x^\ast)\geq \eta$ 
    and if $i,j$ are indices such that 
    \begin{equation}
        \label{eq:minREcond}
        e^{d_H}(R_\varepsilon(x), x^\ast) = \frac{R_\varepsilon(x)_i/x_i^\ast}{R_\varepsilon(x)_j/x_j^\ast},
    \end{equation}
    then using Lemma \ref{lem:Hilbertfar} we have $x_i/x_i^\ast \geq R$, $x_j/x_j^\ast\leq r$. We obtain
    following the steps of \eqref{eq:estimate1} and \eqref{eq:estimate2}
    \begin{align*}
        e^{d_H}(R_\varepsilon(x), x^\ast) =
        \frac{R_\varepsilon(x)_i/x_i^\ast}{R_\varepsilon(x)_j/x_j^\ast} &<
        \frac{ (1-\varepsilon) + \frac{\varepsilon}{\sum_{\nu=1}^n
            \lambda_\nu(x_\nu)^{-1}} \frac{1}{L_2} }{ (1-\varepsilon) +
          \frac{\varepsilon}{\sum_{\nu=1}^n \lambda_\nu(x_\nu)^{-1}}
          \frac{1}{L_1} }
        \ e^{d_H}(x,x^\ast)\\
        & = \left(1 - \varepsilon\ \frac{ \frac{1}{L_1} - \frac{1}{L_2}}{
            \sum_{\nu=1}^n \lambda_\nu(x_\nu)^{-1}(1-\varepsilon) +
            \frac{\varepsilon}{L_1} }\right) \ e^{d_H}(x,x^\ast) \,.\\
\intertext{The term on  the right hand side may be bounded by}
       &<  (1- C_\eta \varepsilon) \ e^{d_H}(x,x^\ast)\,,
\end{align*}
where
\begin{equation*}
    C_\eta = \min_{\varepsilon \in [0,1]} \ \left \{ \frac{ \frac{1}{L_1} - \frac{1}{L_2}}{
            \sum_{\nu=1}^n \lambda_\nu(x_\nu)^{-1}(1-\varepsilon) +
            \frac{\varepsilon}{L_1} } \right \} > 0\,.
\end{equation*}
Note that $C_\eta$ depends on $\eta$ as the choice of $r,R$ is a function
of $\eta$ and these constants in turn determine possible values for $L_1,L_2$.
The final claim follows from a simple application of the logarithm and by
using a standard inequality.  \phantom{a}\hfill\qed
\end{proof}

We also need the following two robustness results. The first concerns the
perturbed averaged system \eqref{eq:incl}, while the second yields a bound
on the worst case behavior of convex combination with arbitrary points in
$\Sigma$.

\begin{lemma}
\label{lem:pertbound}
Let $x^\ast \in \Sigma$ be the unique fixed point of $P$, as described
    in Lemma~\ref{lem:Pfix}.
Consider $\delta^->0$ as defined in \eqref{eq:deltaminusdef}.  
There
exists a constant $K>0$ such that for all $0< \delta < \delta^-$, $\varepsilon \in (0,1)$, all $x\in
P_{\mathrm{co}}(\delta)$, and all $\Delta \in \R^n$
with $e^\top \Delta=0$ and $\|\Delta\|_1 \leq \delta$ we have
    \begin{equation}
        e^{d_H}(R_\varepsilon(x) + \varepsilon\Delta ,x^\ast) - e^{d_H}(R_\varepsilon(x),x^\ast) 
        \leq K \varepsilon \delta\,.
    \end{equation}
\end{lemma}

\begin{proof}
    The assumption on $\delta$ yields that $\conv P(\Sigma) +
    \overline{B}_1(0,2\delta) \subset \ri \Sigma$, see
    \eqref{eq:deltaminusdef}.  By definition we have
    \begin{align*}
        e^{d_H}(R_\varepsilon(x) + \varepsilon \Delta ,x^\ast) -
        e^{d_H}(R_\varepsilon(x),x^\ast) = \frac{ \max_i
          \left\{R_\varepsilon(x)_i + \varepsilon\Delta_i/ x_i^\ast \right
          \}}{\min_j \left\{ R_\varepsilon(x)_j+ \varepsilon\Delta_j/
            x_j^\ast \right \}} -
        \frac{ \max_i \left\{R_\varepsilon(x)_i / x_i^\ast \right \}}{\min_j \left\{ R_\varepsilon(x)_j/ x_j^\ast \right \}} \\
        \intertext{assuming the $i,j$ are chosen so that the maximum, resp. minimum is attained for the perturbed term, we may continue}\\
        \leq \frac{ (R_\varepsilon(x)_i + \varepsilon \delta)
          R_\varepsilon(x)_j / x_i^\ast x_j^\ast - R_\varepsilon(x)_i
          (R_\varepsilon(x)_j - \varepsilon \delta )/ x_i^\ast x_j^\ast}
        {(R_\varepsilon(x)_j - \varepsilon \delta) R_\varepsilon(x)_j/
          x_j^\ast x_j^\ast } = \varepsilon\delta
        \frac{(R_\varepsilon(x)_j + R_\varepsilon(x)_i)/ x_i^\ast
        }{(R_\varepsilon(x)_j - \varepsilon \delta) R_\varepsilon(x)_j/
          x_j^\ast}\,.
    \end{align*}
    To complete the proof, we need to show that the factor of $\varepsilon
    \delta$ in the expression on the right can be uniformly bounded for
    all $x\in P_{\mathrm{co}}(\delta)$. By assumption,
    $P_{\mathrm{co}}(\delta)$ is a compact subset of
    $\ri \Sigma$, so that all entries of $x$ and $R_\varepsilon(x)$ are
    bounded away from $0$. Furthermore, the terms $R_\varepsilon(x)_j -
    \varepsilon \delta$ are bounded away from $0$, because for arbitrary
    indices $j'\neq j$ we have $R_\varepsilon(x) - \varepsilon \delta e_j
    + \varepsilon \delta e_{j'} \in \conv P(\Sigma) +
    \overline{B}_1(0,2\delta) \subset \ri \Sigma$. Thus the factor of
    $\varepsilon\delta$ in the final expression may be bounded by a
    constant, as the denominator is bounded away from $0$.  This constant
    only depends on $\delta^-$.  This proves the claim. \phantom{a}\hfill\qed
\end{proof}

\begin{corollary}
    \label{c:robustdown}
    Let $x^\ast \in \Sigma$ be the unique fixed point of $P$, as described
    in Lemma~\ref{lem:Pfix}. For a given $\eta > 0$, let
    $1>\varepsilon_0>0$ and a $C_\eta>0$ be the constants of
    Corollary~\ref{c:Pconvestimate} such that \eqref{eq:basiccontr2} and
    \eqref{eq:basicconstr3} hold. Let
    \begin{equation}
        \label{eq:deltadef}
        \delta^\ast := \min \left\{ \frac{C_\eta e^\eta}{2K} , \delta^- \right \} 
\,.   
    \end{equation}
    Then for every $0<\delta < \delta^{\ast}$, all $0 < \varepsilon <
    \varepsilon_0$ and all $x \in \conv P(\Sigma) +
    \overline{B}_1(0,\delta)$ and all $\Delta\in \R^n$, $e^\top \Delta=0,
    \|\Delta\|\leq \delta$ we have $R_\varepsilon(x)+ \varepsilon\Delta\in
    P_{\mathrm{co}}(\delta)$ and
    \begin{equation}
        \label{eq:basiccontr2pert}
        d_H(x,x^\ast) \geq \eta \quad \Rightarrow \quad  
        e^{d_H}(R_\varepsilon(x)+ \varepsilon\Delta,x^\ast) < (1 - \frac{C_\eta}{2} \varepsilon) e^{d_H}(x,x^\ast) \,,
    \end{equation}
or equivalently,
\begin{equation}
    \label{eq:basicconstr3pert}
    d_H(x,x^\ast) \geq \eta \quad \Rightarrow \quad  
    d_H(R_\varepsilon(x)+\varepsilon \Delta,x^\ast) < d_H(x,x^\ast) + \log\left(1-\frac{C_\eta}{2}\varepsilon\right)
    < d_H(x,x^\ast) - \frac{C_\eta}{2} \varepsilon\,.   
\end{equation}    
\end{corollary}

\begin{proof}
    The first claim $R_\varepsilon(x)+ \varepsilon\Delta\in
    P_{\mathrm{co}}(\delta)$ is obvious by convexity.
    Under the assumptions we may apply Corollary~\ref{c:Pconvestimate} to
    obtain that $d_H(x,x^\ast) \geq \eta$ implies
    \begin{equation}
e^{d_H}(R_\varepsilon(x),x^\ast) < (1 - C_\eta \varepsilon) e^{d_H}(x,x^\ast) \,.
    \end{equation}
    Thus with an application of Lemma~\ref{lem:pertbound} we obtain
\begin{equation*}
    e^{d_H}(R_\varepsilon(x) + \varepsilon\Delta ,x^\ast) \leq e^{d_H}(R_\varepsilon(x),x^\ast) 
    + K \varepsilon \delta
\leq (1 - \frac{C_\eta}{2} \varepsilon) e^{d_H}(x,x^\ast) + \varepsilon\left( - \frac{C_\eta}{2} e^\eta + K \delta\right)\,.
\end{equation*}
By assumption the last term on the right hand side is negative and we
obtain \eqref{eq:basiccontr2pert}. The final claim
\eqref{eq:basicconstr3pert} is then obvious.  \hfill\qed
\end{proof}

\begin{lemma}
    \label{lem:perturbationbound}
    Let $x^\ast \in \Sigma$ be the unique fixed point of $P$, as described
    in Lemma~\ref{lem:Pfix}.  There exists a constant $C>0$ such that for
    all $x ,y \in \Sigma$ with $x\gg 0$ and for all $\varepsilon\in [0,1)$
    we have
    \begin{equation}
        \label{eq:perutbationbound}
        d_H((1-\varepsilon) x + \varepsilon y, x^\ast ) \leq
        d_H(x,x^\ast) + \log\left( 1+ C\frac{\varepsilon}{1-\varepsilon} \right)\,.
    \end{equation}
    In particular, for any $0< \varepsilon_0 < 1$ there is a constant $C_0$
    such that for all $x ,y \in \Sigma$ with $x\gg 0$ and for all
    $\varepsilon\in [0,\varepsilon_0)$ we have
    \begin{equation}
        \label{eq:perutbationbound2}
        d_H((1-\varepsilon) x + \varepsilon y, x^\ast ) \leq
        d_H(x,x^\ast) +  C_0\varepsilon\,.
    \end{equation}
\end{lemma}

\begin{proof}
    Let $x,y \in \Sigma$ be arbitrary with $x \gg 0$. Then we obtain
    \begin{align*}
        e^{d_H}((1-\varepsilon) x + \varepsilon y, x^\ast ) 
&= \frac{\max_i \{ ((1-\varepsilon)x_i + \varepsilon y_i)/x_i^\ast \} }
       { \min_j \{ ((1-\varepsilon)x_j+ \varepsilon y_j)/x_j^\ast \} } 
\leq 
\frac{\max_i \{ ((1-\varepsilon)x_i + \varepsilon)/x_i^\ast \} }
     { \min_j \{ (1-\varepsilon)x_j/x_j^\ast \}}\\ 
&\leq 
     e^{d_H}(x,x^\ast) + \frac{\varepsilon}{1-\varepsilon}
   \frac{\max_i \{ 1/x_i^\ast \} }{ \min_j \{ x_j/x_j^\ast \} } 
=
e^{d_H}(x,x^\ast) + \frac{\varepsilon}{1-\varepsilon}
  \frac{\max_i \{ 1/x_i^\ast \} }{ \min_j \{ x_j/x_j^\ast \} } 
  \frac{\max_i \{ x_i/x_i^\ast \}}{\max_i \{ x_i/x_i^\ast \}}\\
\intertext{ and using that $\max_i \{ x_i/x_i^\ast \} \geq 1$ and $x_i^\ast\geq x_{\min{}}^\ast$ we obtain}
&\leq
e^{d_H}(x,x^\ast) \left( 1 + \frac{\varepsilon}{1-\varepsilon} \frac{1}{x_{\min{}}^\ast} \right)\,.
    \end{align*}
    The claim \eqref{eq:perutbationbound} now follows by taking the
    logarithm and defining $C$ appropriately. Then
    \eqref{eq:perutbationbound2} follows as $1/(1-\varepsilon)$ is bounded
    on an interval of the form $[0,\varepsilon_0]$ for $\varepsilon_0 <1$.
\hfill \qed
\end{proof}

\section{Proof of the Main Result}
\label{sec:mainproof}

In the following derivation we will make use of a simple fact
concerning sequences of random variables.

\begin{lemma}
    \label{lem:bill}
    Let $\{U_k\}_{k \in \N}$ be a sequence of independent, identically distributed,
    real-valued random variables with well defined expectation $\Expect(U_1) <0$ and finite
    variance $\mathrm{VAR}(U_1) \in \R$.
    %, satisfying 
   % $\Prob(U_1=a) = p_1$ and
    %$\Prob(U_1=b)=p_2$ with $p_1+p_2 =1$.
  Suppose that
   % \[ a< 0 < b \qquad   \Expect(U_1) =: r < 0.
   % \]
  %and  
  $\{ \varepsilon_k \}_{k\in \N}$ is a sequence of positive real numbers that is
    square summable, but not summable. Then
    \begin{equation}
        \label{eq:stochconv1}
        \sum_{k=1}^{L} \varepsilon_k U_k \to - \infty \quad \text{as }L \to \infty\quad \text{a.s.}
    \end{equation}
    Furthermore,
    \begin{equation}
        \label{eq:updown}
        \lim_{\ell\to \infty}\ \sup_{L\geq 0}\ \sum_{k=\ell}^{\ell+L} \varepsilon_k U_k = 0 \quad \text{a.s.}
    \end{equation}
\end{lemma}

\begin{proof} 
Introduce the random sequence $\{V_k\}$ defined by
\[
 V_k= \epsilon_k ( U_k- \bar{U} ),
 \]
 where $\bar{U}:= \Expect(U_1)$. Then, for all $k$,  $\bar{U} =\Expect(U_k)$
and $\Expect(V_k) = 0$.
Also, since the second moment of $U_k$ exists, we may compute
\[{\mathrm{VAR}}(V_k)= \epsilon_k^2 \ {\mathrm{VAR}}(U_k) = \epsilon_k^2 \ 
{\mathrm{VAR}}(U_1) \,.\]
%\epsilon_k^2(\Expect(U_1^2) -\Expect(U_1)^2)\]
 % \[
 % \Expect(V_k^2)  = 
 % \epsilon_k^2(\Expect(U_k^2) -\bar{U}^2) 
 % =\epsilon_k^2(\Expect(U_1^2) -\bar{U}^2) \,;
 % \]
%
  %  As ${\mathrm{VAR}}(\varepsilon_k U_k) = \varepsilon_k^2
  % {\mathrm{VAR}}(U_k)$, the variance of the variables in the series
  %  \eqref{eq:stochconv1} is summable,
%  hence the sequence $\{ \Expect(V_k^2) \}$
  % the variance of the sequence $\{V_k\}$
  By assumption on the sequence $\{ \varepsilon_k \}$ the series
  $\sum_{k=0}^\infty  {\mathrm{VAR}}(V_k)$ converges
  and so by
    \cite[Theorem~22.6]{billingsley2009convergence} 
   % \begin{equation*}
      %  \sum_{k=1}^{\infty} \varepsilon_k (U_k - \Expect(U_k)) 
   % \end{equation*}
   $\sum_{k=1}^\infty V_k$
    converges almost surely to a finite value. 
    Since
 \begin{equation}
 \epsilon_k  U_k =\epsilon_k  \bar{U} +V_k 
 \end{equation}
  we have
    \[
   \sum_{k=1}^L \varepsilon_k  U_k  = \left(\sum_{k=1}^L\varepsilon_k\right) \bar{U}+  \sum_{k=1}^L V_k 
    \]
By assumption,  the positive sequence  $\{\varepsilon_k\}$ is not summable while $\bar{U} <0$; hence
   $\sum_{k=1}^L \varepsilon_k  U_k$ diverges almost surely to $-\infty$.
\\
    
     To prove the second claim, consider
     \begin{equation*}
        % \sum_{k=\ell}^{\ell+L} \varepsilon_k U_k \leq \sum_{k=\ell}^{\ell+L} \varepsilon_k (U_k - \Expect(U_k)) \,.
        %\sum_{k=\ell}^{\ell+L} \varepsilon_k U_k \leq \sum_{k=\ell}^{\ell+L} V_k
   \sum_{k=\ell}^{\ell+L} \varepsilon_k  U_k  =
    \left(\sum_{k=\ell}^{\ell+L}\varepsilon_k\right) \bar{U}+  \sum_{k=\ell}^{\ell+L} V_k 
    \le \sum_{k=\ell}^{\ell+L} V_k 
     \end{equation*}
     Again by \cite[Theorem~22.6]{billingsley2009convergence} the partial
     sums on the right are almost surely partial sums of a convergent
     series. Then the Cauchy criterion says that there are only finitely
     many $\ell\in \N$ such that the sum exceeds a given $C>0$. This shows
     ``$\leq 0$'' in \eqref{eq:updown}.  
Equality follows from the case $L=0$. 
\hfill\qed
\end{proof}

In the proof we also need a continuity result extending
Lemma~\ref{lem:uniformfixed} to the family of Markov chains with fixed
probability $z_0 \in \Sigma$. In the following result we use the notation
$\Prob_{z_0}$ to indicate a probability statement for the Markov chain
\eqref{eq:AIMDfixed} with fixed probability $\lambda = \lambda(z_0)$.

\begin{lemma} 
\label{lem:propvi}
Suppose  that Assumptions  (A1)--(A3) hold.
    Consider the family of Markov chains \eqref{eq:AIMDfixed} with fixed
    probability $\lambda = \lambda(y)$, parametrized by $y\in \Sigma$.
Then, for each  $\bar{\delta} >0 $ and $\theta \in (0,1]$ there exists an
    $m\in \mathbb{N}$ such that for all $y \in \Sigma$ 
    \begin{equation}
            \label{eq:preprobbound}
            \hat{\Prob}_{\lambda(y)}\left( \left\| \bar{S}(m) - P(y)e^\top \right\|_1 > \bar{\delta} \right) < \theta \,.
        \end{equation} 
\end{lemma}

\begin{proof}
    Fix $\varepsilon,\delta >0 $
   and $\hat{y}\in \Sigma$.
    By Lemma~\ref{lem:uniformfixed} there exists an $\hat{m}$ such
    that 
\[
\hat{\Prob}_{\lambda(\hat{y})}\left( \left\| \bar{S}(\hat{m}) - P(\hat{y})e^\top \right\|_1 > \delta \right) < \varepsilon\,.
\]
    Now the map $P$ is continuous by Assumptions (A1) and
    (A3). Furthermore, the  map
    \[
    y \rightarrow \hat{\Prob}_{\lambda(y)}\left( \left\| \bar{S}(\hat{m}) - P(\hat{y})e^\top \right\|_1 > \delta \right)
    \]
    is continuous.
    %probability of the linear operator
   % $\bar{S}(k)$ varies continuously with $z_0$. 
   %Since $P$ is also continuous,
   We obtain that
\begin{equation}
        \label{eq:eq:preprobboundinter}
   \hat{\Prob}_{\lambda(y)}\left( \left\| \bar{S}(\hat{m}) - P(y)e^\top
     \right\|_1 > 2 \delta \right) < 2 \varepsilon\,.
    \end{equation}
  %  \eqref{eq:eq:preprobboundinter} 
  holds on a neighborhood of
    $\hat{y}$. % with the constants $2\delta, 2\varepsilon$ replacing
%    $\delta, \varepsilon$. 
    As $\Sigma$ is compact, it is covered by
a finite number  of such neighborhoods. 
With this argument, and as
    $\varepsilon,\delta$ are arbitrary, we see that there are finitely
    many $\hat{m}_1, \ldots, \hat{m}_N$ in $\mathbb{N}$ such that for every $y \in \Sigma$
    there is an $\hat{m}_j$ such that \eqref{eq:eq:preprobboundinter} holds
    with $\hat{m}=\hat{m}_j$.

    The final claim then follows from an application of Tchebycheff's
    inequality as follows. For $y\in \Sigma$ and $k\in \N$ consider the
    real valued random variable 
    \[
    D(k) = \left\| \bar{S}(k) -
      P(y)e^\top \right\|_1.
      \]
       Note that $0\le D(k) \le 2$, as
    $\bar{S}(k)$ and $P(y)e^\top$ are column
    stochastic. Thus trivially, $\Expect(D(k)^2) \leq 4$. Also, if
    \eqref{eq:eq:preprobboundinter} holds then it follows that
    \begin{equation*}
        \Expect(D(\hat{m})) \leq (1-\varepsilon) \delta + 2\varepsilon\,. 
    \end{equation*}
    We note that the latter inequality is independent of a particular
    $y$ and just depends on the fact that $\hat{m}$ is chosen so that
    \eqref{eq:eq:preprobboundinter} holds.  Fix $\overline{\delta},
    \theta>0$. 
    %With the previous observation 
    Suppose that  
    $\varepsilon,\delta>0$  are chosen such that $(1-\varepsilon) \delta +
    2\varepsilon < \overline{\delta}/{2}$.
    % Given $y \in \Sigma$ we may
    %thus choose a $\hat{m}$ such that \eqref{eq:eq:preprobboundinter} holds
Then it follows from \eqref{eq:eq:preprobboundinter}  that
    \begin{equation}
        \label{eq:expvarX}
        \Expect(D(\hat{m})) < \frac{\overline{\delta}}{2}\,.
    \end{equation}
    Denoting 
    \[
    \Pi(k) = A(k-1) \cdots A(0)
    \]
    for $k\in \N$,
     we have
     \begin{eqnarray*}
     \bar{S}(k) = \frac{1}{k+1}\sum_{j=0}^k \Pi(j)
     \end{eqnarray*}
     Hence, for multiples of $\hat{m}$:
     \begin{eqnarray*}
     \bar{S}(\ell \hat{m}) &=& \frac{1}{\ell \hat{m}+1}\sum_{i=1}^{\ell \hat{m}} \Pi(i)
     = \frac{1}{\ell \hat{m}+1}\sum_{\nu=0}^{\ell-1}\sum_{j=0}^{\hat{m}-1}  \Pi(\nu \ell +j)\\
     &=&
       \frac{1}{\ell \hat{m}+1} \sum_{\nu=0}^{\ell-1}  \sum_{j=0}^{m-1} A(\nu \ell + j) \cdots A(\nu \ell) \Pi(\nu \ell) 
         \end{eqnarray*}
     Hence
    \begin{align*}
        \label{eq:finalsteplempropvi}
        D(\ell \hat{m}) &\leq \frac{1}{\ell} 
        \sum_{\nu=0}^{\ell-1} \left\| \frac{1}{\hat{m}}\sum_{j=0}^{\hat{m}-1} A(\nu \ell + j) \cdots A(\nu \ell) \Pi(\nu \ell) - P(y) e^\top \right\|_1\\
&\leq \frac{1}{\ell} 
        \sum_{\nu=0}^{\ell-1} \left\| \frac{1}{\hat{m}}\sum_{j=0}^{\hat{m}-1} A(\nu \ell + j) \cdots A(\nu \ell)  - P(y) e^\top \right\|_1 \|\Pi(\nu \ell)\|_1 \\
&= \frac{1}{\ell} 
        \sum_{\nu=0}^{\ell-1} \left\| \frac{1}{\hat{m}}\sum_{j=0}^{\hat{m}-1} A(\nu \ell + j) \cdots A(\nu \ell)  - P(y) e^\top \right\|_1 \,.
    \end{align*}
    By the independence assumption on the $A(j)$ we see that the final
    term is the average of $\ell$ independent copies of $D(\hat{m})$, the
    variance of which is bounded by $ 4/\ell$. Let $\{ D_\nu(\hat{m})\}_{\nu \in \mathbb{N}}$
    be a sequence of  independent copies of $D(\hat{m})$.
     It
    follows from Tchebycheff's inequality that for all $\ell \geq
    4\sqrt{2}/(\theta \sqrt{\overline{\delta}})$ we have
    \begin{equation*}
\hat{\Prob}_{\lambda(y)}\left(  \left\| \bar{S}(\ell \hat{m}) -
      P(y)e^\top \right\|_1 > \overline{\delta} \right) \leq \hat{\Prob}_{\lambda(y)}\left( \frac{1}{\ell} \sum_{\nu=0}^{\ell-1} D_\nu(\hat{m}) > \Expect( D(\hat{m})) + \frac{\overline{\delta}}{2} \right) < \theta \,.
    \end{equation*}
    As the previous argument only depends on the validity of
    \eqref{eq:eq:preprobboundinter}, it holds uniformly for all $y \in
    \Sigma$ for which the choice of $\hat{m}$ guarantees
    \eqref{eq:eq:preprobboundinter}.  The proof is completed, by choosing
   $m$ sufficiently large so that it is a  common multiple of $\hat{m}_1,\ldots, \hat{m}_N$. \hfill\qed
\end{proof}

\begin{proof} {\bf (of Theorem \ref{t:convergence})}
    In the proof, we make extensive use of the deterministic system
    discussed in Appendix~\ref{sec:deterministic}. We will show that for
    $T$ sufficiently large the behavior of $\bar{z}(T)$ is well
    approximated by the deterministic system.\\

     We assume that the constants
    $\delta^-, \delta^+$ from \eqref{eq:deltaminusdef},
    \eqref{eq:deltaplus} have been fixed. We will use the notation
    $z(T;z_0)$, resp. $\bar{z}(T;z_0)$ to indicate the initial condition
    for the random variable $z(T)$, resp. its second component vector
    $\bar{z}(T)$. Similarly, the notation $z(T+m;z(T))$ indicates the
    conditioning of $z(T+m)$ on a certain value at time $T$, etc.

    Fix $\eta>0$. We aim to show that almost surely the sample path
    $\bar{z}(T)\in B_1(x^\ast, \eta)$ for all $T$ large enough. As
    $\eta>0$ is arbitrary this will show the claim. 

    To attain our goal, we perform the following sequence of choices:
    \begin{enumerate}[(i)]
    \item For the constant ${\eta}$ pick $\varepsilon_0>0$ and $C_\eta>0$
      according to Corollary~\ref{c:Pconvestimate}, so that
      \eqref{eq:basicconstr3} is satisfied for all $0<
      \varepsilon<\varepsilon_0$.
      \item Let $C_0$ be the constant guaranteed by
        Lemma~\ref{lem:perturbationbound} satisfying
        \eqref{eq:perutbationbound2} for all $0< \varepsilon <
        \varepsilon_0$.
      \item Let $K>0$ be the constant given by Lemma~\ref{lem:pertbound}.
      \item Choose $\delta^\ast$, $\bar{\delta}$ according to \eqref{eq:deltadef}, so that
    \begin{equation}
        \label{eq:deltadefrepeat2}
        \delta^\ast := \min \left\{ \frac{C_\eta e^\eta}{2K} , \delta^- \right \} 
        \quad \text{and} \quad 0 < \bar{\delta} <  \delta^\ast/3 \,,
    \end{equation}
    so that Corollary~\ref{c:robustdown} and
    Lemma~\ref{lem:convinv}\,(iii) are applicable. Let
    $C_{\bar{\delta}}>0$ be the constant guaranteed by
    Lemma~\ref{lem:convinv}\,(iii).
      \item Pick $\theta\in (0,1)$ so that 
        \begin{equation*}
            - (1-\theta) + \theta (1+C_{\bar \delta}) < 0 \,,\quad \text{and} \quad 
            - (1-\theta)C_\eta + \theta C_0 < 0\,.
        \end{equation*}
      \item We now appeal to Lemma~\ref{lem:propvi} to determine the
        length of the (short) averaging period discussed in the preamble.

        Using Lemma~\ref{lem:propvi} and \eqref{eq:corstep1}, pick $m\in
        \N$ such that for all $y \in \Sigma$ the Markov chain
        \eqref{eq:AIMDfixed} with fixed probability $\lambda =
        \lambda(y)$ satisfies for all $T\in \N$   that
        \begin{equation}
            \label{eq:probbound}
            \hat{\Prob}_{\lambda(y)}\left( \left\| \frac{1}{m} \sum_{j=1}^m x(T+j) - P(y) \right\|_1 > \bar{\delta} \right) 
            < \frac{\theta}{2} \,.
        \end{equation}
      \item Pick $T_0\in \N$ such that for all $T\geq T_0$ we have
        $m/(T+m) < \varepsilon_0$ and so that for the Markov
        chain~\eqref{eq:AIMD2infav} with place dependent probabilities we
        have 
        %for the average of the first component 
        that
        \begin{equation}
            \label{eq:probboundz}
            \Prob_{x_0}\left( \left\| \frac{1}{m} \sum_{j=1}^{m} 
            x(T+j) - P(\bar{x}(T)) \right\|_1 > \bar{\delta} \right) < \theta \,.
        \end{equation}
        This is possible as this inequality is a perturbed version of
        \eqref{eq:probbound}: indeed, with increasing $T$ the variation of
        $\bar{x}(T+j), 0\leq j \leq m$ (i.e. in the first $m$ steps after
        time $T$) becomes arbitrarily small. More precisely, for
        $j=1\ldots,m$ we have by definition
\begin{equation*}
   \|\bar{x}(T) - \bar{x}(T+j)\|_1 \leq \frac{2m}{T} \,.   
\end{equation*}
Thus as $T\to \infty$ the place dependent probabilities of the Markov
chain~\eqref{eq:AIMD2infav} that are considered on the interval $[T,T+m]$
converge to the fixed probabilities $\lambda(\bar{x}(T))$. The claim then
follows from \eqref{eq:probbound} by continuity of the probability functions $\lambda_i$ (see (A1)).
    \end{enumerate}
    
    Let $T\geq T_0$, so that by construction $\varepsilon := m/(T+m) <
    \varepsilon_0$. 
    We will study the evolution of the value $\bar{x}(T+k
    m) \mapsto \bar{x}(T+(k+1)m)$, $k\in \N$. This is given by
    \begin{align*}
        \bar{x}(T+(k+1)m) &= \frac{T+ k m+1}{ T+ (k+1)m+1 } \bar{x}(T + k m) +
        \frac{m}{ T+ (k+1)m+ 1} \left( \frac{1}{m} \sum_{j=1}^{m}
          x(T+k m + j ) \right)\,.
    \end{align*}
    For ease of notation we define
    % $P_{\mathrm{co}}(\bar{\delta}):= \conv (P(\Sigma) + \overline{B}_1(0,\bar{\delta}))$, 
    \[
    \tau(k) := T + k m
    \qquad \text{and} \qquad 
       \varepsilon_k :=  \frac{m}{ T+ (k+1)m+ 1}
       \]
 so that the previous equation can be expressed as\footnote{This chain is not a typical stochastic
        approximation chain due to the interdependence between $x$
      and $\bar x$---cf. \eqref{eq:AIMD2infav} \eqref{eq:probdef}. 
      It is similar
    in form to the two-timescales process of
    \cite[Chapter~6]{Borkar08}, but does not satisfy the convergence conditions therein.}
    \begin{align}
                \bar{x}(\tau(k+1)) &= (1-\varepsilon_{k}) \bar{x} (\tau(k)) +
         \varepsilon_{k} \left( \frac{1}{m} \sum_{j=1}^{m}
          x(\tau(k)+j) \right)\,.
    \end{align}
    At this point the reader should recognize the structure of the
    discrete iteration we have analyzed in Section~\ref{sec:deterministic}
    and notice that by (vii) we have a high probability that
    $\bar{x}(\tau(k+1))$ is close to $R_{\varepsilon_{k}}(\bar{x}
    (\tau(k)))$.\\

 Note that the constants have
    been chosen so that both Lemma~\ref{lem:convinv} and
    Corollary~\ref{c:robustdown} are applicable.  We will use the
    estimates obtained in Lemma~\ref{lem:convinv} to show that
    trajectories starting in $\Sigma$ will reach the set 
    $P_{\mathrm{co}}(2\bar{\delta})$
    in a finite number of steps. We
    then show that for trajectories starting in the strict superset
    $P_{\mathrm{co}}(3\bar{\delta})$
    the estimates of
    Corollary~\ref{c:robustdown} yield that we reach the set
    $B_H(x^\ast,\eta)$ again in a finite number of steps; almost surely.\\

    {\parindent0pt \bf Step 1:} More precisely, we will first show that (the
    first hitting time)
    \begin{equation*}
        \sigma_1 := \min \{ k \in \N \;;\; \bar{x}(\tau(k)) \in
P_{\mathrm{co}}(2\bar{\delta})\ \} 
    \end{equation*}
    is almost surely finite. Obviously, if $\bar{x}(T)\in
    P_{\mathrm{co}}(2\bar{\delta})$ 
    there is nothing to show. 
    Appealing
    to Lemma~\ref{lem:convinv}\,(i) and the choice made in (vii), we have
    that if $\distone(\bar{x}(\tau(k)), P_{\mathrm{co}}(\bar{\delta})) > \bar{\delta}$, then
    \begin{equation*}
        \Prob_{x_0}\Bigl( \distone(\bar{x}(\tau(k+1)), P_{\mathrm{co}}(\bar{\delta})) \leq (1-\varepsilon_{k})\ \distone(\bar{x}(\tau(k)), P_{\mathrm{co}}(\bar{\delta})) \Bigr) \geq 1 - \theta\,.
    \end{equation*}
    In the complementary event, which happens with probability of at most
    $\theta$ we have by Lemma~\ref{lem:convinv} that
    \begin{equation*}
        \distone(\bar{x}(\tau(k+1)), P_{\mathrm{co}}(\bar{\delta})) \leq (1+ \varepsilon_{k} C_{\bar \delta})\,
        \distone(\bar{x}(\tau(k)), P_{\mathrm{co}}(\bar{\delta})) 
    \end{equation*}
    Combining these two observations we see that for $\tau(k) <
    \sigma_1$ we have that
\begin{equation}
\label{eq:decay0}
    \distone(\bar{x}(\tau(k)), P_{\mathrm{co}}(\bar{\delta})) \leq \left( \prod_{\ell=1}^k a_\ell\right) \,
    \distone(\bar{x}(T), P_{\mathrm{co}}(\bar{\delta})), 
\end{equation}
where $a_\ell, \ell \in \N$ is a random variable that has the value
$(1-\varepsilon_\ell)$ with probability $1-\theta$ and the value
$(1+\varepsilon_\ell C_\delta )$ with probability $\theta$. 
By
construction the random variables $a_\ell$ are independent, as the bounds
obtained do not depend on the particular sample path of the Markov
chain.

To be able to apply Lemma~\ref{lem:bill} we first note that for all $\ell$
large enough we have $\log(1+\varepsilon_\ell C_\delta ) <
\varepsilon_\ell (1+C_\delta)$. By the choice of $\theta$ in (vi), we
obtain for all $\ell$ large enough that $\Expect( \log a_\ell) \leq
\left(-(1-\theta) + \theta (1+C_\delta)
\right)\varepsilon_\ell$. Lemma~\ref{lem:bill} thus implies that
$\sum_{\ell=1}^k \log a_\ell \to -\infty$, almost surely. Thus almost
surely we have
\begin{equation*}
    \lim_{k\to \infty} \distone(\bar{x}(\tau(k)), P_{\mathrm{co}}(\bar{\delta})) = 0\,,
\end{equation*}
provided that $\bar{x}(\tau(k)) \notin 
P_{\mathrm{co}}(2\bar{\delta})$ for all $k$. This is of course impossible, and
so almost surely $\bar{x}(\tau(k)) \in 
P_{\mathrm{co}}(2\bar{\delta})
$ for a finite $k$.\\

{\parindent0pt \bf Step 2:} Similarly, if $\bar{x}(\tau(k)) \in 
P_{\mathrm{co}}(3\bar{\delta})
$, then by Corollary~\ref{c:robustdown} and the
choice made in (vii) we have
\begin{equation*}
    \Prob_{x_0}\left( d_H (\bar{x}(\tau(k+1)), x^\ast) < 
d_H(\bar{x}(\tau(k)),x^\ast) - \frac{C_\eta}{2} \varepsilon_{k} \right) \geq 1 - \theta\,.
\end{equation*}
On the other hand with probability of at most $\theta$ we have by
Lemma~\ref{lem:perturbationbound} that
    \begin{equation}
        \label{eq:leaving}
        d_H(\bar{x}(\tau(k+1)), x^\ast ) \leq
        d_H(\bar{x}(\tau(k)),x^\ast) +  C_0\varepsilon_{k}\,.
    \end{equation}
    In a similar fashion to the first step, as long as
    $\bar{x}(\tau(k)) \in P_{\mathrm{co}}(3\bar{\delta})
$
    and $d_H(\bar{x}(\tau(k)),x^\ast) > \eta$, we have
\begin{equation}
\label{eq:decay1}
    d_H(\bar{x}(\tau(\sigma_1(\bar{z}(T)+k)), x^\ast ) \leq
        d_H(\bar{x}(\sigma_1(\bar{z}(T)+k)),x^\ast) + \sum_{\ell=1}^k b_\ell\,,
\end{equation}
where $b_\ell$ is a random variable that takes the value
$-\varepsilon_\ell C_\eta$ with probability $(1-\theta)$ and the value
$\varepsilon_\ell C_0$ with probability $\theta$. As before,
Lemma~\ref{lem:bill} ensures that $\sum_{\ell} b_\ell$ diverges to
$-\infty$, almost surely.  Note that it is always possible to leave the
set $P_{\mathrm{co}}(3\bar{\delta})
$
with a small probability. In this case Step~1 can be applied again, so
that we re-enter the set $P_{\mathrm{co}}(2\bar{\delta})
$, almost surely. Now by \eqref{eq:decay0} the process of entering
$P_{\mathrm{co}}(2\bar{\delta})$ and subsequently leaving $P_{\mathrm{co}}(3\bar{\delta})$ requires that for some partial sum we have
\begin{equation*}
    \sum_{k=\ell}^{\ell + L} \log(a_k) \geq \bar{\delta}\,.
\end{equation*}
By Lemma~\ref{lem:bill}, with probability $1$, this happens only a finite
number of times. Consequently, almost surely a sample path will reach 
$B_H(x^\ast,\eta)$.\\

{\parindent0pt \bf Step 3:} Finally, to obtain almost sure convergence, we
need to show that almost surely
\begin{equation}
    \label{eq:asconv}
   \bar{x}(\tau(k)) \in B_H(x^\ast,\eta) \,,\quad \text{for all $k$ large enough.} 
\end{equation}
To this end we repeat the choices made in (i) - (vii) for the value
$\eta/2$. Thus we can conclude that almost surely a sample path enters
$B_H(x^\ast,\eta/2)$. If we assume that the sample path leaves
$B_H(x^\ast,\eta)$ at some later time, then again by Steps 1 and 2 it will
almost surely re-enter $B_H(x^\ast,\eta/2)$. The question is thus whether
it is possible that infinitely often the sample path exits the ball
$B_H(x^\ast,\eta)$ given that it was previously within the ball
$B_H(x^\ast,3\eta/4)$. 
In view of \eqref{eq:decay1} this amounts to saying
that
\begin{equation*}
    \sum_{k=\ell}^{\ell+L} b_k > \frac{\eta}{4}
\end{equation*}
for pairs $(\ell,L)\in \N^2$ with arbitrarily large $\ell$.  By
Lemma~\ref{lem:bill} this almost surely does not happen. This shows
\eqref{eq:asconv}. The proof is complete by noting that the small
variations of $\bar{x}$ on the intervals $\tau(k), \ldots, \tau(k+1)$ do
not destroy stability. 
Indeed, if $\bar{x}(\tau(k)) \in B_H(x^\ast,\eta/2)$
for all $k$ large enough, then also $\bar{x}(\tau(k)+j) \in B_H(x^\ast,\eta)$ for $j=1,\ldots,m$, provided $k$ is large enough.
~\hfill\qed
\end{proof}

\section{Acknowledgment}
The authors thank Ronald Fagin for helpful comments on an earlier version
of this paper.

%\bibliographystyle{abbrv}
%\bibliography{sonja_bib}

\end{document}